\documentclass[12pt, a4paper]{article}
\usepackage[utf8]{inputenc}
\usepackage[margin=2cm]{geometry}

\usepackage{xcolor}
\definecolor{myblue}{rgb}{0.1 0.1 0.6}
\usepackage{hyperref}
\hypersetup{
   colorlinks=true,
   linkcolor=myblue,
   citecolor=myblue,
   urlcolor=myblue
}

\usepackage{amsmath,amssymb,amsthm}
\usepackage[shortlabels]{enumitem}
\usepackage[T1]{fontenc}
\usepackage{beton}
\usepackage[euler-digits, euler-hat-accent]{eulervm}
\DeclareFontSeriesDefault[rm]{bf}{sbc} 

\usepackage{tikz-cd}

\newtheorem{theorem}{Theorem}[section]
\newtheorem{proposition}[theorem]{Proposition}
\newtheorem{lemma}[theorem]{Lemma}
\theoremstyle{definition}
\newtheorem{corollary}[theorem]{Corollary}
\newtheorem{example}[theorem]{Example}
\newtheorem{remark}[theorem]{Remark}
\newtheorem{definition}[theorem]{Definition}

\numberwithin{equation}{section}

\newcommand{\mc}{\mathcal}

\newcommand{\id}{\operatorname{id}}

\newcommand{\lin}{\operatorname{lin}}

\newcommand{\ip}[2]{\langle{#1},{#2}\rangle}
\newcommand{\op}{{\operatorname{op}}}
\newcommand{\Tr}{{\operatorname{Tr}}}
\newcommand{\im}{{\operatorname{Im}\,}}
\newcommand{\vnten}{\bar\otimes}

\newcommand{\at}{\operatorname{At}}
\newcommand{\rel}{\operatorname{Rel}}

\newcommand{\clos}{\overline{\phantom{d}}}
\newcommand{\rankone}[2]{|{#1}\rangle \langle{#2}|}
\newcommand{\rep}{\operatorname{Rep}}
\newcommand{\Hilb}{\operatorname{Hilb}}

\begin{document}
\title{Quantum relations in the general setting: composition and adjacency operators}
\author{Matthew Daws}
%\date{February 2026}
\maketitle

\begin{abstract}
Quantum relations in the sense of Weaver are $M'$-bimodules, for a von Neumann algebra $M$, these generalising actual relations on a set $X$ when $M=\ell^\infty(X)$.  Similarly, relations between two sets can be generalised as bimodules over the commutants of two algebras.  We make an explicit study of this idea, developing some tools to check that constructions are well-defined.  Motivation comes from Kornell's concept of a Quantum Set (for algebras which are sums of matrix algebras), and we find that $*$-homomorphisms correspond to certain quantum relations, extending unpublished work of Kornell.  We find a functor from completely positive maps to quantum relations, related to the idea of taking a noisy communication channel and reducing it to its underlying ``relation''.  As with Quantum Graphs, at least in finite-dimensions, quantum relations correspond to ``adjacency operators'', certain CP maps depending on a choices of faithful functional on the algebras.  We develop some tools to deal with the non-Schur-idempotent case, and show links with our functor from CP maps, and work of Verdon.  We explicitly compute the adjacency operator of a $*$-homomorphism.
\end{abstract}

\section{Introduction}

\emph{Noncommutative}, or \emph{Quantum}, \emph{Graphs} are a noncommutative formalisation of a finite, simple (possibly with loops) graph, in the language of operator algebras (and so that over a finite-dimensional, commutative algebra, we exactly recover the usual notion of a finite graph).  They arose in parallel developments from Quantum Information Theory, \cite{duan2009superactivationzeroerrorcapacitynoisy, DSW_ZeroError} for example, from Weaver's work \cite{Weaver_QuantumRelations, Weaver_QuantumGraphs} which is more directly in a ``non-commutative'' framework, and from category-theoretic quantum mechanics, \cite{MRV_Compositional, Verdon_CovQuantumCombs} for example.  Such approaches are essentially equivalent, see for example our survey \cite[Section~5]{daws_quantum_graphs}.

In Weaver's formulation, quantum graphs are special types of \emph{quantum relations}, although always over a single von Neumann algebra $M$.  It is not hard to imagine how to adapt Weaver's definition to a quantum relation between two von Neumann algebras, and indeed, this is implicit in the literature, e.g.\@ the unpublished \cite{kornell2015quantumfunctions}.  In this paper, we make a more explicit study of this notion, making links with existing notions for finite-dimensional von Neumann algebras, \cite{Kornell_QuantumSets, kornell2026quantumgraphshomomorphisms, Verdon_CovQuantumCombs}, and showing how the different ``pictures'' of quantum graphs have quantum relation versions.

A quantum relation over a von Neumann algebra $M \subseteq\mc B(H)$ is a weak$^*$-closed, $M'$-bimodule $V\subseteq\mc B(H)$, \cite[Section~2]{Weaver_QuantumRelations}.  By ``bimodule'' we mean simply that $x\in V, a,b\in M' \implies axb\in V$; this should not be confused with Hilbert $C^*$-modules, which make a natural appearance later.  That we have a bimodule over the commutant, and that $V$ seems to depend upon the Hilbert space $H$ which $M$ acts on, are linked in the sense that if $M\subseteq\mc B(H_1)$, then quantum relations inside $\mc B(H)$, and inside $\mc B(H_1)$, are in an order-preserving bijection, \cite[Theorem~2.7]{Weaver_QuantumGraphs}.  So there is a no real dependence on $H$, but one must be careful when making definitions that the definition is well-defined (see the discussion at the end of Section~\ref{sec:qrs}): this is an important concern of ours in this paper.

To formulate the notion of a quantum relation between two algebras $M$ and $N$, one could either look at relations over $M\oplus N$ which are supported in a ``corner'', Definition~\ref{defn:qrel_2}, or more directly look at $N'$-$M'$-bimodules inside $\mc B(H,K)$, where $M\subseteq\mc B(H), N\subseteq\mc B(K)$, Definition~\ref{defn:qrel_1}.  It is not hard to see that these are equivalent, but it seems less clear to us that we \emph{immediately} get the desired ``invariance'' from the representing Hilbert spaces $H,K$.  We check the necessary details in Section~\ref{sec:qrs}, and develop some results to check ``well-defined-ness'' in general.  Kornell, in \cite{Kornell_QuantumSets} and later work, looked at \emph{quantum sets}, which can be realised as von Neumann algebras of (possibly infinite) direct-sums of matrix algebras, and this was our initial motivation: we start Section~\ref{sec:qrs} by motivating the general case from the matrix case.

A central result of \cite{Kornell_QuantumSets} is to show that certain quantum relations (the \emph{coinjective} ones) correspond to $*$-homomorphisms between the associated von Neumann algebras of matrices.  We were interested in giving a purely von Neumann algebraic proof of this; in fact, this was already done in the unpublished \cite{kornell2015quantumfunctions}, where a very brief indication is given as to how Hilbert C$^*$-module theory gives the result.  We think it useful to give slightly more detail, which we do in Section~\ref{sec:quantum_functions}, and then to develop the necessary theory to fully characterise e.g.\@ when a $*$-homomorphism is unital, or injective, in terms of the quantum relation, Theorem~\ref{thm:funcs_to_HMs}.  This shows that the category of quantum relations contains the category of $*$-homomorphisms, for all von Neumann algebras.

Guided by quantum graph theory, one might now consider if there was a notion related to ``quantum confusability graphs'' or ``quantum adjacency operators''.  These turn out to be intimately related, a fact perhaps foreseen by Verdon in \cite{Verdon_CovQuantumCombs} (here Verdon works with a diagrammatic calculus, and $2$-categories, so the meeting point with this paper is the finite-dimensional situation; see Section~\ref{sec:CP_to_QR_again} for a fuller account).  In Section~\ref{sec:channels_to_relations} we generalise the notion of taking the ``support'' of a communication channel to realise a relation (Example~\ref{eg:classical_channel}) to give a procedure which associates a UCP map $\theta$ with a quantum relation $V^\theta$, in fact, a functor, Proposition~\ref{prop:cp_to_qr_composition}.  As with quantum graphs, \cite{Weaver_QuantumGraphs}, there is a notion of a ``pullback'' of a relation along a CP map, and this can be nicely expressed in the category of relations, Section~\ref{sec:pullbacks}.  The contemporaneously released preprint \cite{kornell2026quantumgraphshomomorphisms} explores some equivalent ideas for Quantum Sets in Kornell's sense.

In Section~\ref{sec:adj_ops}, we turn to adjacency operators, which for quantum graphs are a generalisation of the adjacency matrix of a graph.  Here again, the natural extension to operators between two algebras can be made.  We work here with the conventions of \cite{daws2025quantumgraphsinfinitedimensionshilbertschmidts, Wasilewski_Quantum_Cayley} and look at CP (or ``real'') adjacency operators.  For quantum graphs, one only looks at Schur idempotent adjacency operators (or equivalently, projections in the algebra $M\vnten M^\op$).  However, here it is important to consider more general operators, and we develop some theory in this direction.  This now makes links with \cite{Verdon_CovQuantumCombs} more clear, and we show that the functor $\theta \mapsto V^\theta$ from Section~\ref{sec:channels_to_relations}, and the association between adjacency operators $A$ and quantum relations $V$, is almost the same, Proposition~\ref{prop:A_vs_theta}.

We are now in a position to give an account of Verdon's construction \cite[Proposition~3.12]{Verdon_CovQuantumCombs} showing that all (finite-dimensional) quantum graphs arise as confusability graphs of UCP maps (generalising the matrix case, \cite{duan2009superactivationzeroerrorcapacitynoisy}).  We do this in a purely operator-algebraic language, Proposition~\ref{prop:Verdon_All_QG_from_UCP}, and then give a new result characterising the symmetric quantum relations $S$ which arise from CP maps (the UCP map case corresponding to the case when $1\in S$), Theorem~\ref{thm:QG_from_CP}; we use here calculation techniques developed previously in the paper.

Composition of quantum relations manifests for adjacency operators as composition (for the usual product) which in general does not preserve Schur idempotency, hence motivating our earlier work on non-idempotent operators.  We briefly consider how the natural order on relations is reflected at the adjacency level (this seems complicated), and then show how to explicitly compute the adjacency operator of a quantum function (that is, $V^\theta$ arising from a $*$-homomorphism $\theta$).  In the non-tracial case, this is surprisingly intricate.  At a number of points, we give examples to illustrate open problems: these are mostly questions around which relations $V=V^\theta$ arise from certain classes of CP map $\theta$.

\subsection{Notation}\label{sec:notation}

We work with standard notation for von Neumann algebras, following \cite{TakesakiII} for example.  Our inner-products are linear in the right variable.  Given an algebra $M$, let $M^\op$ be the opposite algebra to $M$, say with a typical element $x^\op$ for $x\in M$, and multiplication $x^\op y^\op = (yx)^\op$.  We use bra-ket notation in a limited way: for $\xi\in H$ denote $\langle \xi |$ the linear map $H\to\mathbb C; \eta \mapsto (\xi|\eta)$, and by $|\xi\rangle$ the adjoint, the linear map $\mathbb C\to H; \lambda\mapsto \lambda\xi$.  Hence any rank-one operator on $H$ is of the form $\rankone{\xi}{\eta}$.  Given a Hilbert space $H$, let $\overline{H}$ be the conjugate Hilbert space.  By Riesz–Fréchet, for any $\mu\in H^*$, the Banach space dual, there is $\xi\in H$ such that $\mu = \langle\xi|$, and then the induced map $\mu \mapsto \overline\xi$ gives an isometric isomorphism $H^* \to \overline H$.  Denote by $HS(H)$ the Hilbert--Schmidt operators on $H$, which is a Hilbert space isomorphic to $H\otimes\overline H$ for the map $\rankone{\xi}{\eta} \mapsto \xi \otimes \overline\eta$; we often make this identification without further comment.  See Section~\ref{sec:qrs} for the notion of a ``transpose'' of an operator $T^\top \in\mc B(\overline H)$ for $T\in\mc B(H)$.

\subsection{Acknowledgements}

I thank Andre Kornell, Dominic Verdon and Makoto Yamashita for useful conversations related to aspects of this work.

\section{Quantum sets and general quantum relations}\label{sec:qrs}

We give a quick summary of the notion of a \emph{quantum set} from \cite{Kornell_QuantumSets}.  Following \cite[Section~2]{Kornell_QuantumSets} a quantum set $\mc X$ is a family of finite-dimensional Hilbert spaces indexed by a set $\at(\mc X)$, the \emph{atoms} of $\mc X$.  Write $X \propto \mc X$ to indicate that $X$ is an atom of $\mc X$.  In a more ``non-commutative geometry'' viewpoint, we alternatively identify $\mc X$ with the von Neumann algebra
\[ L^\infty(\mc X) = \ell^\infty\text{-}\bigoplus_{X\propto \mc X} \mc B(X) \cong
\ell^\infty \text{-} \bigoplus_{X\propto\mc X} \mathbb M_{\dim(X)}, \]
the von Neumann algebra of finite-dimensional matrix algebras.%  Similarly one can define $C_0(\mc X)$.

Following \cite[Section~3]{Kornell_QuantumSets}, a \emph{binary relation} $R$ from a quantum set $\mc X$ to a quantum set $\mc Y$ is an assignment of a subspace $R(X,Y) \subseteq \mc B(X,Y)$ for each $X\propto\mc X, Y \propto\mc Y$.  Write $\rel(\mc X; \mc Y)$ for the set of (binary) relations from $\mc X$ to $\mc Y$.
Composition of two relations is defined by setting
\[ (S\circ R)(X,Z) = \lin \{ sr : \exists\, Y\propto\mc Y, s\in S(Y,Z), r\in R(X,Y) \}. \]
The identity relation is $1_{\mc X}$ with components $1_{\mc X}(X,Y) = \mathbb C 1_X$ when $X=Y$, and $\{0\}$ otherwise.  There is a natural order: $R \leq S$ when $R(X,Y) \subseteq S(X,Y)$ for all $X,Y$.

There are two adjoint-type operations.  This paper uses functional-analysis language, and so we follow the notational norms here, which unfortunately clash with \cite{Kornell_QuantumSets}.  Denote by $X^*$ the Banach space dual of $X$, which we can identify with the conjugate Hilbert space $\overline X$, Section~\ref{sec:notation}.  As usual, for $T\colon X\to Y$ let $T^*$ be the Hilbert space adjoint, and denote by $T^t$ the Banach space adjoint defined by $T^t(f) = f\circ T$ for $f\in Y^* = \mc B(Y,\mathbb C)$.  Identify $f\in Y^*$ with $\overline {y_0}$, so $f(y) = (y_0|y)$ for $y\in Y$.  Then $(T^t(f))(x) = f(T(x)) = (y_0|T(x)) = (T^*(y_0)|x)$ for $x\in X$, and so $T^t(f)$ is identify with $\overline{T^*(y_0)} = T^\top\overline{y_0}$.  Hence $T^t$ is identified with $T^\top \colon \overline Y \to \overline X$.  

Given a relation $R$ from $\mc X$ to $\mc Y$, define $R^*$ a relation from $\mc Y$ to $\mc X$ by $R^*(Y,X) = \{ r^* : r\in R(X,Y) \}$.  Define $\overline{\mc X}$ to be $\{ \overline X : X\propto \mc X \}$, and then define $\overline R$ to be the relation from $\overline{\mc Y}$ to $\overline{\mc X}$ given by $\overline R(\overline Y, \overline X) = \{ r^\top : r\in R(X,Y) \}$.  This could also be defined using the Banach space dual, given the remarks before.

We hence obtain a (dagger compact) category $\mathsf{qRel}$, see \cite[Theorem~3.6]{Kornell_QuantumSets}.

We wish to re-cast (some elements of) this category is a more abstract, functional analytic language, in particular in a ``coordinate free'' way which only uses the algebra $L^\infty(\mc X)$.  The obvious way to do this is to adapt Weaver's notion of a quantum relation to the case of two algebras.  There are hints of this idea in the literature, see work of Kornell for example, but we do not know a place where some key details are explicitly worked out, so we proceed to do this here.  Adapting \cite[Definition~2.1]{Weaver_QuantumRelations}, we make the following definition.

\begin{definition}\label{defn:qrel_1}
Let $M\subseteq\mc B(H), N\subseteq\mc B(K)$ be von Neumann algebras.  A \emph{quantum relation} from $M$ to $N$ is a weak$^*$-closed subspace $V \subseteq \mc B(H,K)$ which is a $N'$-$M'$-bimodule: if $x\in V, a\in N'\subseteq\mc B(K), b\in M'\subseteq\mc B(H)$ then $axb\in V$.
\end{definition}

For a quantum set $\mc X$ define $L^2(\mc X) = \bigoplus \{ HS(X) : X\propto\mc X \}$, the Hilbert space direct sum.  Also define $\ell^2(\mc X) = \bigoplus \{ X : X\propto\mc X\}$, again the Hilbert space direct sum.  For $X\propto\mc X$, as $\mc B(X)$ acts naturally on $X$ and on $HS(X)$, we see that $L^\infty(\mc X)$ can be regarded as a von Neumann algebra in $\mc B(\ell^2(\mc X))$ and $\mc B(L^2(\mc X))$.  The space $L^2(\mc X)$ is the GNS space of any faithful positive functional on $L^\infty(\mc X)$, so we regard it as more canonical than $\ell^2(\mc X)$.

We regard operators on $\ell^2(\mc X)$ as (infinite) matrices $(T_{X,Y})$ where $T_{X,Y} \in\mc B(Y,X)$ for $X,Y\propto \mc X$.  Then $L^\infty(\mc X)$ consists of the diagonal matrices (and such a matrix defines a bounded operator exactly when $\sup_X \|T_{X,X}\| <  \infty$).  For $X\in\mc X$ let $1_X$ denote the minimal central projection in $L^\infty(\mc X)$ which is the identity on the $X$ factor, and $0$ elsewhere.  Let $T=(T_{X,Y}) \in L^\infty(\mc X)'$, so $1_X T = T1_X$, for each $X$, and so $T$ is seen to be diagonal.  Then $T_{X,X}$ commutes with every member of $\mc B(X)$, and so $T_{X,X} = t_X 1_X$ for some scalar $t_X$.  We conclude that $L^\infty(\mc X)' \cong \ell^\infty(\at(\mc X))$.

\begin{lemma}
Let $\mc X,\mc Y$ be quantum sets and regard $L^\infty(\mc X) \subseteq \mc B(\ell^2(\mc X))$ and the same for $\mc Y$.  A quantum relation from $L^\infty(\mc X)$ to $L^\infty(\mc Y)$ is of the form
\begin{equation}
V = \lin \{ V_{X,Y} \}\clos^{w^*} \label{eq:form_V}
\end{equation}
where for each $X\in\mc X, Y\in\mc Y$ we have that $V_{X,Y} \subseteq \mc B(X,Y)$ is a subspace, and we regard $V_{X,Y} \subseteq \mc B(\ell^2(\mc X), \ell^2(\mc Y))$ in the obvious way.  Any such collection of subspaces gives rise to a quantum relation in this way.
\end{lemma}
\begin{proof}
We have just observed that $L^\infty(\mc X)' \cong \ell^\infty(\at(\mc X))$ and similarly for $\mc Y$.  Thus a quantum relation from $L^\infty(\mc X)$ to $L^\infty(\mc Y)$ is a weak$^*$-closed subspace $V$ which is a $\ell^\infty(\at(\mc Y))$-$\ell^\infty(\at(\mc X))$-bimodule.  Given such a $V$, let $V_{X,Y} = 1_Y V 1_X \subseteq \mc B(X,Y) \subseteq \mc B(\ell^2(\mc X), \ell^2(\mc Y))$.  Then $V_{X,Y} \subseteq V$ be the bimodule property, and so the weak$^*$-closed linear span of the $V_{X,Y}$ is also a subspace of $V$.  Let $\at(\mc X)_{00}$ be the finite subsets of $\at(\mc X)$, ordered by inclusion to form a direct set.  For $\alpha \in \at(\mc X)_{00}$ let $1_\alpha = \sum_{X\in\alpha} 1_X$, and observe that $1_\alpha \to 1$ $\sigma$-weakly in $L^\infty(\mc X)$.  Thus for any $x\in V$ we have that $1_\alpha x 1_\beta$ approximates $x$ weak$^*$, and clearly $1_\alpha x 1_\beta \in \lin \{ V_{X,Y} \}$.  Thus we have equality as in \eqref{eq:form_V}.  Conversely, if $V$ is of the form \eqref{eq:form_V} then $V$ is weak$^*$-closed, and is a bimodule, so a quantum relation.
\end{proof}

\begin{corollary}\label{corr:relations_relations}
Let $\mc X,\mc Y$ be quantum sets.  There is a canonical bijection between $\rel(\mc X; \mc Y)$ and (Weaver) quantum relations from $L^\infty(\mc X) \subseteq \mc B(\ell^2(\mc X))$ to $L^\infty(\mc Y) \subseteq \mc B(\ell^2(\mc Y))$.
\end{corollary}

Notice that we write $V_{X,Y} \subseteq \mc B(X,Y)$ to be consistent with the notation for the component $R(X,Y)$ of a relation $R$, but that this is inconsistent with the notation for a(n operator-valued) matrix.

Let us see what the category structure corresponds to:
\begin{itemize}
    \item The identity $1_{\mc X}$ is naturally identified with $\ell^\infty(\at(\mc X))$, and so corresponds to the relation $L^\infty(\mc X)'$.
    \item Composition is simply operator composition: $V\circ W = \lin\{ v\circ w :v\in V, w\in W\}\clos^{w^*}$.
    \item Order is simply inclusion.
    \item The adjoint corresponds to $V^* = \{ v^* : v\in V \}$.
    \item As $\mc B(\overline X) \cong \mc B(X)^\top$ is an anti-homomorphism, we see that $L^\infty(\overline{\mc X})$ can be canonically identified with $L^\infty(\mc X)^\op$, the opposite algebra.  Then $R^\top$ corresponds to the operation $V^\top = \{ v^\top : v\in V \}$.
\end{itemize}

We next look at how we can regard these generalised relations as a special type of relation in Weaver's sense.  Given two von Neumann algebras $M,N$ consider the direct sum $M\oplus N$.  If $M \subseteq\mc B(H), N\subseteq\mc B(K)$ then $M\oplus N \subseteq \mc B(H\oplus K)$ as ``diagonal matrices'', where we view $\mc B(H\oplus K)$ as an algebra of $2\times 2$ matrices in the usual way.  Notice then that $(M\oplus N)' = M' \oplus N'$.  Then $\mc B(H,K)$ can be identified with the bottom-left corner in this matrix picture, which in turn is identified with $1_K \mc B(H\oplus K) 1_H$.  As $1_H = 1_M \in M \cap M'$ and similarly for $N$, we have motivated the following, which it is now easy to see is equivalent to Definition~\ref{defn:qrel_1}.

\begin{definition}\label{defn:qrel_2}
A quantum relation from $M$ to $N$ is a quantum relation on $M\oplus N$, say $V$, with $1_N V 1_M = V$.
\end{definition}

Notice that we stated Corollary~\ref{corr:relations_relations} for $L^\infty(\mc X)$ acting on $\ell^2(\mc X)$.  What about other choices, for example, $L^2(\mc X)$?  For a single algebra, Weaver shows in \cite[Theorem~2.7]{Weaver_QuantumRelations} that quantum relations are essentially independent of the choice of Hilbert space our algebra acts on.  It is not, to us, immediately clear that this will also hold for relations over two algebras, so we proceed to check the details.  The following constructions for $*$-homomorphisms will also be used extensively in much of the rest of the paper.

Rather than follow Weaver's argument directly (but see Proposition~\ref{prop:Weaver_corr_is_ours}), we follow our treatment, compare \cite[Lemmas~7.3, 7.4]{daws_quantum_graphs}.  Fix $M\subseteq\mc B(H), N\subseteq\mc B(K)$, and let $V\subseteq\mc B(H\oplus K)$ be a quantum relation on $M\oplus N$.  Given a Hilbert space $L$, define the normal unital injective $*$-homomorphism $\pi \colon M\oplus N \to \mc B((H\oplus K)\otimes L); x\mapsto x\otimes 1_L$, and let $V_\pi = V \vnten \mc B(L) = \{ v\otimes x : v\in V, x\in\mc B(L) \}\clos^{w^*}$.  Then $V \mapsto V_\pi$ is an order-preserving bijection between quantum relations on $M\oplus N$ and quantum relations on $\pi(M\oplus N)$, \cite[Lemma~7.3]{daws_quantum_graphs}.

\begin{lemma}
The map $V\mapsto V_\pi$ preserves the property that $V = 1_N V 1_M$.
\end{lemma}
\begin{proof}
As $V$ is a bimodule over $M'\oplus N'$, that $V = 1_N V 1_M$ is equivalent to $V \subseteq 1_N V 1_M$, as the other inclusion follows from the bimodule property.  Clearly $V \subseteq 1_N V 1_M$ implies $V_\pi \subseteq 1_N V_\pi 1_M$.

For the converse, first note that the map $x \mapsto 1_N x 1_M$ is idempotent, weak$^*$-continuous, and restricts to a map on $V\otimes \mc B(L)$, the algebraic tensor product.  It follows that $1_N V_\pi 1_M = 1_N V 1_M \vnten \mc B(L)$, compare for example \cite[Lemma~5.5]{Daws_OneParam}.  Suppose that $V_\pi = 1_N V_\pi 1_M$ but that, towards a contradiction, $1_N V 1_M$ is a proper subspace of $V$.  By Hahn-Banach, there is $\omega\in\mc B(H\oplus K)_*$ and $v\in V$ with $\ip{v}{\omega}=1$ but $\ip{1_N u 1_M}{\omega}=0$ for each $u\in V$.  Then $(\omega\otimes\id)(1_N V 1_M \otimes \mc B(L)) = \{0\}$ and so $(\omega\otimes\id)(1_N V_\pi 1_M) = \{0\}$ by weak$^*$-continuity, and the remark before.  However, then $\{0\} = (\omega\otimes\id)(V_\pi) = \omega(V) \mc B(L)$ which contradicts $\ip{v}{\omega}=1$.
\end{proof}

Now let $\theta \colon M\oplus N \to \mc B(L)$ be a normal unital $*$-homomorphism, so using the structure theory for such maps, compare before \cite[Lemma~7.4]{daws_quantum_graphs} for example, there is $L_0$ and an isometry $u \colon L \to (H\oplus K)\otimes L_0$ with $u \theta(x) = \pi(x) u$ for $x\in M\oplus N$, with $\pi$ as before, formed using $L_0$.  Set $V_\theta = u^* V_\pi u$.  Then $V \mapsto V_\theta$ is order preserving and maps quantum relations on $M\oplus N$ to quantum relations on $\theta(M\oplus N)$, \cite[Lemma~7.4]{daws_quantum_graphs}.

\begin{lemma}\label{lem:HM_mapping_preserves_qr}
The map $V \mapsto V_\theta$ preserves the property that $V = 1_N V 1_M$.
\end{lemma}
\begin{proof}
Let $P = \pi(M\oplus N) = (M\oplus N)\otimes 1 \subseteq \mc B((H\oplus K)\otimes L_0)$.  By the previous lemma, we need only show that $V \mapsto u^*Vu$ preserves the property, for quantum relations $V$ on $P$.  Recall that as $u$ intertwines $\theta$ and $\pi$, we have that $uu^* \in P' = (M'\oplus N')\vnten\mc B(L_0)$.  Also, $u 1_M = u \theta(1_M) = \pi(1_M) u = 1_M u$ and for $1_N$.  So if $V \subseteq 1_N V 1_M$ then $u^* V u \subseteq u^* 1_N V 1_M u = 1_N (u^*Vu) 1_M$.  Again we obtain equality by the bimodule property.
\end{proof}

\begin{corollary}\label{corr:indep_rep}
Quantum relations from $M$ to $N$ are independent of the representation chosen.
\end{corollary}
\begin{proof}
Let $M\subseteq\mc B(H_1)$ and $\mc B(H_2)$ and similarly for $N$ on $K_1, K_2$.  Let $\theta_{1,2} \colon M\oplus N \to \mc B(H_2\oplus K_2)$ be the resulting $*$-homomorphism, which has inverse $\theta_{2,1} \colon M \oplus N \to \mc B(H_1\oplus K_1)$, where we first consider $M\oplus N$ as acting ambiently on $H_1\oplus K_1$, and then secondly as acting on $H_2\oplus K_2$.  There is an order-preserving bijection between quantum relations $V_1\subseteq \mc B(H_1\oplus K_1)$ on $M\oplus N$, and $V_2\subseteq \mc B(H_2\otimes K_2)$; indeed, $V_2 = (V_1)_{\theta_{1,2}}$ and $V_1 = (V_2)_{\theta_{2,1}}$.  If $V_1 = 1_N V_1 1_M$ then the lemma applied to $\theta_{1,2}$ shows that $V_2 = 1_N V_2 1_M$, while the converse follows by using $\theta_{2,1}$ instead.
\end{proof}

\begin{remark}\label{rem:changing_base_spaces}
For later reference, let us give a more concrete description of $V_\theta$.  Typically we are interested in the case when $\theta_M \colon M \to \mc B(H_1)$ and $\theta_N \colon M \to \mc B(K_1)$ are normal unital $*$-homomorphisms, giving $\theta = \theta_M \oplus \theta_N \colon M \oplus N \to \mc B(L)$ with $L=H_1\oplus K_1$.  Choose isometries $u_M \colon H_1 \to H \otimes L_1$ and $u_N \colon K_1 \to K\otimes L_2$ dilating $\theta_M, \theta_N$ respectively, so with $L_0 = L_1 \oplus L_2$ we get
\begin{align*}
u \colon  &  L = H_1\oplus K_1 \to (H\oplus K) \otimes L_0 \cong (H \otimes L_1) \oplus (H\otimes L_2) \oplus (K \otimes L_1) \oplus (K\otimes L_2), \\
&(\xi,\eta) \mapsto (u_M(\xi), 0, 0, u_N(\eta)).
\end{align*}
Then $u^*(z\otimes 1)u = \theta(z)$ for $z\in M\oplus N$, and so $V_\theta = u^*(V\vnten\mc B(L_0))u$, thought of as a quantum relation over $M\oplus N$.

As a quantum relation from $M$ to $N$, we look at the $\mc B(H_1, K_1)$ corner.
As we take weak$^*$-closures, it is enough to consider a rank-one operator in $\mc B(L_0)$.  Let $\gamma_i=(\alpha_i,\beta_i)\in L_0=L_1\oplus L_2$ for $i=1,2$, and set $t = \rankone{\gamma_1}{\gamma_2}\in\mc B(L_0)$.  For $\xi\in H_1, \eta\in K_1$, and $v\in V\subseteq\mc B(H,K)$, we are hence interested in
\begin{align*}
\big( (0,\eta) \big| u^*(v \otimes t) u (\xi,0) \big)
&= \big( (0,0,0,u_N(\eta)) \big| (v\otimes \rankone{(\alpha_1,\beta_1)}{(\alpha_2,\beta_2)}  (u_M(\xi),0,0,0) \big) \\
&= \big( (1\otimes \langle \beta_1|)u_N(\eta)  \big|  v (1\otimes \langle\alpha_2|)u_M(\xi) \big) \\
&= \big( \eta \big| u_N^* (v\otimes\rankone{\beta_1}{\alpha_2}) u_M(\xi) \big).
\end{align*}
We conclude that $V_\theta = u_N^* (V\vnten\mc B(L_1, L_2)) u_M$.  Here of course $V\vnten\mc B(L_1,L_2)$ is the weak$^*$-closure of $V\otimes \mc B(L_1,L_2)$, which we can think of as a complemented subspace of $\mc B((H\oplus K)\otimes (L_1\oplus L_2))$.
\end{remark}

Let us take the opportunity to complete the discussion we started in \cite{daws_quantum_graphs}.  Weaver's argument in \cite[Theorem~2.7]{Weaver_QuantumRelations} is that when $M \subseteq\mc B(H)$ and $\theta \colon M \to \mc B(K)$ is an injective unital normal $*$-homomorphism, then\footnote{We give an essentially self-contained proof of this claim, and some bibliographic notes, in \url{https://github.com/MatthewDaws/Mathematics/tree/master/Isomorphism-VN}} for some suitably large $L$ we have that $M \otimes 1$ and $\theta(M)\otimes 1$ are spatially isomorphic: there is a unitary $w \colon K \otimes L \to H\otimes L$ with $w^*(x\otimes 1)w = \theta(x)\otimes 1$ for each $x\in M$.  Then quantum relations $V\subseteq \mc B(H)$ over $M$, and $W\subseteq \mc B(K)$ over $\theta(M)$ are related exactly when $w^*(V\vnten\mc B(L))w = W\vnten\mc B(L)$.

\begin{proposition}\label{prop:Weaver_corr_is_ours}
The quantum relation $W\subseteq\mc B(K)$ associated to $V$ is exactly $V_\theta$.
\end{proposition}
\begin{proof}
We continue with the notation just introduced.  Let $\xi_0\in L$ be a unit vector and define $u \colon K \to H\otimes L$ by $u(\xi) = w(\xi\otimes\xi_0)$ so for $x\in M$ we have $u^*(x\otimes 1)u = (\id\otimes\omega_{\xi_0})(w^*(x\otimes 1)w) = \theta(x)$.  In particular, $u^*u=1$, and so we have a dilation of $\theta$.  By the uniqueness claim of \cite[Lemma~7.4]{daws_quantum_graphs} we have that
\[ V_\theta = u^*(V\vnten\mc B(L))u
= (\id\otimes\omega_{\xi_0})(w^*(V\vnten\mc B(L))w)
= (\id\otimes\omega_{\xi_0})(W\vnten\mc B(L))
= W. \]
For the final equality, notice that $(\id\otimes\omega_{\xi_0})(W\vnten\mc B(L)) \subseteq W$ by weak$^*$-continuity of the slice-map, and that the other inclusion is clear.  Thus $V_\theta = W$ as claimed.
\end{proof}

The map $U\mapsto U_\theta$ developed above does not need that $\theta$ is injective, and so is a mild generalisation of \cite[Theorem~2.7]{Weaver_QuantumRelations} (of course, we do not obtain a bijection in general!)  The following shows that this ``base change'' operation commutes with composition.

\begin{proposition}\label{prop:base_change_to_comp}
For $i=1,2,3$ let $M_i \subseteq \mc B(H_i)$ be a von Neumann algebra, and $\theta_i \colon M_i \to \mc B(K_i)$ a unital normal $*$-homomorphism.  Let $U\subseteq\mc B(H_1, H_2), V \subseteq \mc B(H_2,H_3)$ be quantum relations from $M_1$ to $M_2$, respectively, $M_2$ to $M_3$.  Then $V_{\theta_2\oplus\theta_3} \circ U_{\theta_1\oplus\theta_2} = (V\circ U)_{\theta_1\oplus\theta_3}$.
\end{proposition}
\begin{proof}
Let $u_i \colon K_i \to H_i\otimes L_i$ dilate $\theta_i$, so $U_{\theta_1\oplus\theta_2} = u_2^*(U \vnten \mc B(L_1, L_2)) u_1$, by Remark~\ref{rem:changing_base_spaces}.  Similarly, $V_{\theta_2\oplus\theta_3} = u_3^*(V \vnten \mc B(L_2,L_3)) u_2$.  Thus
\begin{align*}
& V_{\theta_2\oplus\theta_3} \circ U_{\theta_1\oplus\theta_2}  \\
&= \lin\big\{ u_3^*(v \otimes t_1) u_2 u_2^*(u\otimes t_2) u_1 : v\in V, u\in U, t_1\in \mc B(L_2,L_3), t_2\in  \mc B(L_1, L_2) \big\}\clos^{w^*}  \\
&= \lin\big\{ u_3^*(v \otimes t_1)(u\otimes t_2) u_1 : v\in V, u\in U, t_1\in \mc B(L_2,L_3), t_2\in  \mc B(L_1, L_2) \big\}\clos^{w^*}  \\
&= u_3^*((V\circ U) \vnten \mc B(L_1,L_3)) u_1 = (V\circ U)_{\theta_1\oplus\theta_3},
\end{align*}
here using that $u_2u_2^* \in M_2' \vnten \mc B(L_2)$, and $U$ is a left $M_2'$-module (or that $V$ is a right $M_2'$-module), and that $\mc B(L_2,L_3) \circ \mc B(L_1,L_2) = \mc B(L_1,L_3)$, which requires a routine argument (approximation by finite-rank operators) to check.
\end{proof}

For the notion of relation from Definition~\ref{defn:qrel_1}, we might formally think of a relation as really being an equivalence class, where for each pair $H,K$ with $M\subseteq\mc B(H), N\subseteq\mc B(K)$ we have some representative of $V$.  We shall not be this formal, but it is important to always check that definitions are well-defined, that is, we really are specifying one equivalence class.

We hence obtain a category: objects are von Neumann algebras and morphisms are quantum relations.  It is perhaps not completely obvious that composition is associative, and as we use similar ideas later, let us give the details.  Given quantum relations $U,V,W$ we claim that $U\circ (V\circ W)$ equals
\[ X = \lin\{ uvw : u\in U, v\in V, w\in W \} \clos^{w^*}. \]
Obviously $X \subseteq U\circ (V\circ W)$.  Conversely, we show that for $u\in U, y\in V\circ W$ we have $uy\in X$; then $U\circ (V\circ W)\subseteq X$ as $X$ is a weak$^*$-closed subspace.  However, as $X$ is weak$*$-closed and a subspace, and multiplication is separately weak$^*$-continuous, we can approximate $y$ by elements of the form $vw$ for $v\in V, w\in W$, and the claim follows.  Similarly $X = (U\circ V)\circ W$ and so in particular $(U\circ V)\circ W = U\circ (V\circ W)$.

The identity morphism on $M$ is the quantum relation $M'$, for if $V$ is a quantum relation from $M$ to $N$ then $V \circ M' \subseteq V$ be the bimodule property, but as $1\in M'$, obviously then $V = V\circ M'$.

We shall not study further category-theoretic properties of this category, compare \cite{Kornell_QuantumSets}, though this would be interesting to do.

\section{Functions between quantum sets}\label{sec:quantum_functions}

Re-writing and generalising definitions from \cite[Section~4]{Kornell_QuantumSets} in our notation, we have the following.

\begin{definition}\label{defn:functions}
Let $M,N$ be von Neumann algebras, and let $V$ be a quantum relation from $M$ to $N$.  Then $V$ is said to be \emph{coinjective} when $V\circ V^* \subseteq N'$ and \emph{cosurjective} when $V^*\circ V \supseteq M'$.  When $V$ satisfies both of these conditions, we say that $V$ is a \emph{function}; a coinjective relation is a \emph{partial function}.  Similarly, $V$ is \emph{injective} when $V^*\circ V \subseteq M'$ and \emph{surjective} when $V\circ V^* \supseteq N'$.
\end{definition}

We immediately check that these notions are well-defined.

\begin{lemma}\label{lem:com_under_dilation}
Let $M\subseteq\mc B(H)$ be a von Neumann algebra and $\theta_M \colon M \to \mc B(H_1)$ a unital normal $*$-homomorphism, dilated as $\theta_M(x) = u_M^*(x\otimes 1)u_M$ for some isometry $u_M \colon H_1 \to H \otimes L$ say.  Then $x \in \theta_M(M)'$ if and only if $u_M x u_M^* \in M' \vnten \mc B(L)$ and so $\theta_M(M)' = u_M^*(M'\vnten\mc B(L))u_M$.
\end{lemma}
\begin{proof}
For $x\in\mc B(H_1)$ we have the chain of implications
\begin{align*}
x \in \theta_M(M)'
& \iff
x u_M^*(y\otimes 1)u_M = u_M^*(y\otimes 1)u_M x  \qquad (y\in M) \\
& \implies
u_M x u_M^*(y\otimes 1)u_Mu_M^* = u_M u_M^*(y\otimes 1)u_M x u_M^*  \qquad (y\in M)  \\
& \implies
u_M x u_M^* (y\otimes 1) = (y\otimes 1) u_M x u_M^*  \qquad (y\in M)  \\
& \implies
x u_M^* (y\otimes 1)u_M = u_M^* (y\otimes 1) u_M x  \qquad (y\in M)
\end{align*}
using that $u_M^*u_M=1$ and $u_M u_M^* \in M'\vnten\mc B(L)$.  So we have equivalences throughout, and hence $x\in\theta_M(M)'$ if and only if $u_Mxu_M^* \in (M\otimes 1)' = M'\vnten\mc B(L)$.
So, given $y\in M'\vnten\mc B(L)$ we have that $x = u_M^* y u_M$ satisfies $u_M x u_M^* = (u_Mu_M^*) y (u_Mu_M^*) \in M'\vnten\mc B(L)$ so $x\in\theta_M(M)'$; conversely, for $x\in\theta_M(M)'$ we have that $x = u_M^*u_M x u_M^* u_M \in u_M^*(M'\vnten\mc B(L))u_M$.
\end{proof}

\begin{proposition}
Let $V$ be a quantum relation from $M$ to $N$.  The choice of Hilbert spaces $H,K$ upon which $M,N$ act does not change whether $V$ is (co)injective or (co)surjective.
\end{proposition}
\begin{proof}
Let $\theta_M \colon M \to \mc B(H_1)$ and $\theta_N \colon N \to \mc B(K_1)$ be faithful representations, and set $\tilde V = V_{\theta_M\oplus\theta_N}$, so we wish to show that $V$ is cosurjective if and only if $\tilde V$ is, and so forth.  As $\tilde V^* = V^*_{\theta_N\oplus\theta_M}$, Proposition~\ref{prop:base_change_to_comp} shows that $\tilde V^* \circ \tilde V = (V^*\circ V)_{\theta_M\oplus\theta_M}$.  If $M' \subseteq V^*\circ V$ then $(M')_{\theta_M\oplus\theta_M} \subseteq \tilde V^* \circ \tilde V$, but $(M')_{\theta_M\oplus\theta_M} = u_M^*(M'\vnten\mc B(L_1))u_M = \theta_M(M)'$ by Lemma~\ref{lem:com_under_dilation}.  (This last point also follows from \cite[Theorem~2.7]{Weaver_QuantumRelations}.)  So $V$ cosurjective implies that $\tilde V$ is cosurjective, and swapping the roles of $H,H_1$ and $K,K_1$ shows the converse.  The other cases are analogous.
\end{proof}

We aim to generalise \cite[Theorem~7.4]{Kornell_QuantumSets} which shows that over algebras of the form $L^\infty(\mc X), L^\infty(\mc Y)$, quantum functions biject with unital normal $*$-homomorphism $L^\infty(\mc Y) \to L^\infty(\mc X)$.  In fact, this was already done by Kornell in the unpublished work \cite{kornell2015quantumfunctions}, with a direct argument, and also by the claim (without further detail) that the result actually follows directly from \cite[Theorem~4.2.7]{SG_QuanStocProcess_NCG} or \cite[Proposition~6.12]{Rieffel_mortia_cstar_wstar}, that is, by using some results on Hilbert $C^*$-modules.  For the record, and as we wish to prove a bit more, we shall sketch how this argument proceeds.

Fix von Neumann algebras $M\subseteq\mc B(H)$ and $N\subseteq\mc B(K)$.  We do not assume that homomorphisms are unital.

\begin{proposition}\label{prop:defn_V_theta}
Let $\theta \colon N \to M$ be a normal $*$-homomorphism and define
\[ V^\theta = \{ v\in\mc B(H,K) : yv = v\theta(y) \ (y\in N) \}. \]
Then $V^\theta$ is a quantum relation from $M$ to $N$ which is coinjective.
\end{proposition}
\begin{proof}
Clearly $V^\theta$ is a weak$^*$-closed subspace.  For $a\in M', b\in N'$ and $v\in V^\theta$, for $y\in N$ we have $y (bva) = b (yv) a = b v \theta(y) a = (bva) \theta(y)$, showing that $bva\in V^\theta$.  Hence $V^\theta$ is a bimodule, that is, a quantum relation.

For $u^*\in V^{\theta *}$, as $\theta$ is a $*$-homomorphism, we have that $\theta(y) u^* = u^*y$ for $y\in N$.  So given also $v\in V^\theta$, for $y\in N$ we have $v u^* y = v \theta(y) u^* = y vu^*$, showing that $vu^*\in N'$.  Hence $V^{\theta} \circ V^{\theta *} \subseteq N'$, that is, $V^{\theta}$ is coinjective.
\end{proof}

We now consider a general coinjective quantum relation $V$ from $M$ to $N$.  Consider $V^*$, which is injective, and a right $N'$-module.  As $V^*$ is injective, we can define an $N'$-valued inner-product by $(u^*|v^*) = uv^*\in N'$ for $u^*,v^*\in V^*$.  Hence $V^*$ is a Hilbert $C^*$-module over $N'$: this means that $V^*$ is a right $N'$-module, and has an $N'$-valued inner-product compatible with this module structure; see, for example, \cite{Lance_HilbModsBook} for more on this.

\begin{proposition}\label{prop:selfdual}
The module $V^*$ is self-dual, meaning that bounded $N'$-linear maps $V^*\to N'$ are given by taking an inner-product with an element of $V^*$
\end{proposition}
\begin{proof}
The module $V^*$ is a concrete Hilbert module, as it is by definition represented on a Hilbert space.  Such objects are well-studied, but we are not aware of a canonical reference.  To show that $V^*$ is self-dual, it is quickest to invoke  \cite[Lemma~8.5.4(1)]{BlecherLeMerdy_Book} which tells us that $V^*$ is self-dual, because the $N'$-valued inner-product is separately weak$^*$-continuous in each variable, which is clearly the case here.
\end{proof}

Much of the literature (e.g. \cite{paschke_inner_prod_mods, Rieffel_mortia_cstar_wstar}) works with self-dual modules, but the book \cite{SG_QuanStocProcess_NCG} works instead with topological condition; but in practise, the proofs apply to any weak$^*$-closed concrete Hilbert module, and so do apply to $V^*$.  In fact these approaches are really equivalent, \cite{Skeide_vnmods_ints}.

The following is the result we need from the book \cite{SG_QuanStocProcess_NCG}; we give a proof using Rieffel's work \cite{Rieffel_mortia_cstar_wstar} for completeness.

\begin{theorem}[{\cite[Theorem~4.2.7]{SG_QuanStocProcess_NCG}}]\label{thm:SG}
Let $H,K$ be Hilbert spaces, let $N_0 \subseteq\mc B(K)$ be a von Neumann algebra, and let $X\subseteq\mc B(K,H)$ be a Hilbert $C^*$-module over $N_0$ which is weak$^*$-closed, and satisfies that $\lin\{ x(\xi) : x\in X, \xi\in K \}$ is dense in $H$.  There is a unique normal unital $*$-homomorphism $\theta \colon N_0' \to \mc B(H)$ with
\[ X = V^{\theta *} = \{ x\in\mc B(K,H) : xa' = \theta(a')x \ (a'\in N_0') \}. \]
\end{theorem}
\begin{proof}
As $N_0 \subseteq \mc B(K)$, we see that $K$ is a ``generator'', see \cite[Proposition~1.3]{Rieffel_mortia_cstar_wstar}.  As $X$ is self-dual (same proof as Proposition~\ref{prop:selfdual}) by \cite[Proposition~6.12]{Rieffel_mortia_cstar_wstar} there is a Hilbert space $L$ with a normal unital $*$-homomorphism $\rho\colon N_0' \to \mc B(L)$ with $X$ isomorphic (as a Hilbert $C^*$-module) to
\[ V^{\rho *} = \{ x \in \mc B(K,L) : xa' = \rho(a')x \ (a'\in N_0') \}. \]
Indeed, $L = F_X(K)$ which, in the language of \cite[Chapter~4]{Lance_HilbModsBook} is the \emph{interior tensor product} of $X$ with $K$, namely the completion of the algebraic tensor product for the pre-inner-product given on elementary tensors by
\[ (x\otimes\xi|y\otimes\eta) = (\xi|x^*y\eta) = (x\xi|y\eta) \qquad (x,y\in X, \xi,\eta\in K). \]
As $X\subseteq\mc B(K,H)$ we get this simple form for the inner-product, and by the density hypothesis, we conclude that $L$ is isometric with $H$.  Thus identify $L$ with $H$ and set $\theta = \rho$.
\end{proof}

\begin{corollary}[Kornell]\label{corr:kornells_bijec}
Let $V\subseteq\mc B(H,K)$ be a quantum function from $M$ to $N$.  Then $V = V^\theta$ for a unique unital normal $*$-homomorphism $\theta \colon N \to M$.
\end{corollary}
\begin{proof}
By hypothesis, $V\circ V^* \subseteq N'$, so we can regard $V^*\subseteq\mc B(K,H)$ as a Hilbert $C^*$-module over $N_0 = N'$.  Also $M' \subseteq V^*\circ V$.  Suppose that $\xi_0\in H$ with $(\xi_0|x\xi)=0$ for all $x\in V^*, \xi\in K$.  Then certainly $0 = (\xi_0|x^*y\xi)=0$ for each $x,y\in V, \xi\in H$ and so $0 = (\xi_0|x\xi)$ for each $x\in M', \xi\in H$.  Taking $x=1, \xi=\xi_0$ shows that $\xi_0=0$, and so the density condition of Theorem~\ref{thm:SG} holds.  There is a unital normal $*$-homomorphism $\theta \colon N = N_0' \to \mc B(H)$ with $V = V^\theta$.  Furthermore, for $a\in N, b\in M', x\in V^*$ we have that $xa = \theta(a)x$, but also $bx\in V^*$ so $bxa = \theta(a)bx$, and hence $\theta(a)bx = b\theta(a)x$.  By the density condition, $\theta(a) b = b \theta(a)$, for all $a,b$, and so $\theta$ maps into $M'' = M$, as required.
\end{proof}

We shall finesse this argument to deal with non-unital homomorphisms, and relations which are only assumed to be coinjective.  We first study various properties of $V^\theta$.  Indeed, the definition of $V^\theta$ depends upon the choices of $H,K$ with $M\subseteq\mc B(H), N\subseteq\mc B(K)$, so firstly we check that we have correctly defined an ``equivalence class'', see the discussion at the end of Section~\ref{sec:qrs}.

\begin{proposition}\label{prop:changing_base_space}
Let $\theta_M \colon M \to \mc B(H_1)$ and $\theta_N \colon N \to \mc B(K_1)$ be normal injective unital $*$-homomorphisms, and let $\theta \colon N \to M$ be a normal $*$-homomorphism.  Let $\theta_1 \colon \theta_N(N) \to \theta_M(M)$ be induced by $\theta$, and then form $V^\theta$ and $V^{\theta_1}$.  Then $V^{\theta_1}$ is exactly the quantum relation induced by $V^\theta$ when transported along the homomorphism $\theta_M \oplus \theta_N \colon M\oplus N \to \mc B(H_1\oplus K_1)$.
\end{proposition}
\begin{proof}
We again use the notation of Remark~\ref{rem:changing_base_spaces}, and so we wish to show that $V^{\theta_1} = u_N^* (V^\theta \vnten \mc B(L_1,L_2)) u_M$.  By definition, $v \in V^{\theta_1}$ when $\theta_N(y) v = v \theta_M(\theta(y))$ for $y\in N$, that is, $u_N^* (y\otimes 1) u_N v = v u_M^*(\theta(y)\otimes 1)u_M$.  Arguing as in the proof of Lemma~\ref{lem:com_under_dilation}, we find that
\[ V^{\theta_1} = \{ v\in\mc B(H_1,K_1) : (y\otimes 1) u_N v u_M^* = u_N v u_M^* (\theta(y)\otimes 1) \ (y\in N) \}. \]
Again we regard $\mc B(H_1,K_1)$ as a subspace of $\mc B(H_1 \oplus K_1)$, and so forth, and so think of $u_N v u_M^*$ as acting on $(H_1\oplus K_1) \otimes L$.  Hence we can slice out by a normal functional, and comparing the resulting condition with the definition of $V^\theta$, we conclude that $v\in V^{\theta_1}$ exactly when $(\id\otimes\omega)(u_N v u_M^*) \in V^\theta$ for each $\omega\in\mc B(L)_*$.  By \cite[Remark~1.5]{Kraus_SliceMap}, as the bounded operators on a Hilbert space have property $S_\sigma$, it follows that
\[ V^{\theta_1} = \{ v \in \mc B(H_1, K_1) : u_N v u_M^* \in V^\theta \vnten \mc B(L_1, L_2) \} = u_N^* (V^\theta \vnten\mc B(L_1,L_2)) u_M, \]
for the second equality arguing exactly as in the second part of Lemma~\ref{lem:com_under_dilation}.
\end{proof}

For a $*$-homomorphism $\theta \colon N \to M$, as $\theta(1)$ is a projection in $M$ commuting with $\theta(N)'$, we have that $\theta(1)\theta(N)' = \theta(1) \theta(N)' \theta(1)$ is a non-unital von Neumann algebra (that is, a weak$^*$-closed self-adjoint subalgebra of $\mc B(H)$).

\begin{lemma}\label{lem:com_theta_image}
For a normal $*$-homomorphism $\theta\colon N\to M$ we have that $\theta(1) \theta(N)' = V^{\theta*} \circ V^{\theta}$.
\end{lemma}
\begin{proof}
By Proposition~\ref{prop:changing_base_space} we can change the space $N$ acts on without changing $V^{\theta *}\circ V^\theta$.  So by replacing $K$ by $K\otimes L$ for a suitable $L$, we may suppose that there is $r\colon H \to K$ with $r \theta(y) = yr$ for each $y\in N$, and with $r^*r = \theta(1) \in M$.  Given $x\in \theta(N)'$ set $u = xr^* \colon K \to H$.  For $y\in N$ we have $uy = xr^*y = x \theta(y) r^* = \theta(y) xr^* = \theta(y)u$, and so $u^*\in V^\theta$.  As $uu^* = xr^*rx^* = x \theta(1) x^* = \theta(1) xx^*$, and as $V^{\theta*} \circ V^{\theta}$ is a linear space, we conclude that  $\theta(1) \theta(N)' \subseteq V^{\theta*} \circ V^{\theta}$.

For the other inclusion, let $v\in V^\theta$ so $v^*y = \theta(y)v^*$ for $y\in N$, and hence $v^* = \theta(1)v^*$.  For $u,v\in V^\theta$ we have $v^*u\theta(y) = v^* y u = \theta(y) v^*u$ and so $v^*u \in \theta(N)'$, hence $v^*u = \theta(1)v^*u \in\theta(1)\theta(N)'$.
\end{proof}

\begin{corollary}\label{corr:density_cond}
For a normal $*$-homomorphism $\theta\colon N\to M$ we have that $\overline\lin\{ v^*\xi : v\in V^\theta, \xi\in K\} = \theta(1)(H)$.
\end{corollary}
\begin{proof}
By Lemma~\ref{lem:com_theta_image}, we have that $\overline\lin V^{\theta *}K \supseteq \lin (V^{\theta*}\circ V^\theta)H = \lin \theta(1)\theta(N)'H = \theta(1)(H)$ as $\theta(N)'$ leaves $\theta(1)(H)$ invariant.  Conversely, for any $v\in V^\theta$ we have that $v^* = v^*1 = \theta(1)v^*$ and so $\lin V^{\theta *}K \subseteq \theta(1)(H)$.
\end{proof}

For any coinjective $V$, as $V\circ V^* \subseteq N'$, and always $V\circ V^*$ is an $N'$-bimodule, we have that $V\circ V^*$ is a weak$^*$-closed ideal, and so $V\circ V^* = zN'$ for some central projection $z\in N\cap N'$, \cite[Proposition~II.3.12]{TakesakiI}.

\begin{lemma}\label{lem:kernal_theta}
Let $\theta \colon N \to M$ be a normal $*$-homomorphism, and let $z\in N\cap N'$ be the central projection with $V^\theta \circ V^{\theta *} = zN'$.  Then $\ker(\theta) = (1-z)N$.
\end{lemma}
\begin{proof}
If $\theta(y)=0$ then for $v\in V^\theta$ we have $yv = v\theta(y)=0$.  Conversely, suppose $\theta(y)\not=0$, so also $\theta(y^*)\theta(1) = \theta(y)^* \not=0$, and so by Corollary~\ref{corr:density_cond}, there is $v\in V^\theta$ with $\theta(y)^*v^*\not=0$.  We conclude that $\theta(y)=0$ if and only if $yv = v\theta(y)=0$ for every $v\in V^\theta$.

So $y\in\ker\theta$ implies that $yvu^* = 0$ for all $u,v\in V^\theta$, while conversely, if in particular $y v v^* = 0$ for all $v$, then $y v v^* y^* = 0$ for all $v$, so $yv=0$ for all $v$, so $\theta(y)=0$.  Thus $y\in\ker\theta$ if and only if $y zN' = y (V^{\theta}\circ V^{\theta^*}) = \{0\}$, which is equivalent to $yz=0$, which is equivalent to $y \in (1-z)N$ as claimed.
\end{proof}

We now state our main result, which generalises various equivalences between relations and homomorphisms from \cite{Kornell_QuantumSets} to arbitrary von Neumann algebras.

\begin{theorem}\label{thm:funcs_to_HMs}
There is a bijection between normal $*$-homomorphisms $\theta \colon N \to M$ and coinjective quantum relations $V$ from $M$ to $N$, given by $\theta \mapsto V^\theta$.  Under the bijection:
\begin{enumerate}[(1)]
\item\label{thm:funcs_to_HMs:one}
   unital homomorphisms correspond to quantum functions;
\item\label{thm:funcs_to_HMs:two}
   injective homomorphisms correspond to surjective quantum relations;
\item\label{thm:funcs_to_HMs:three}
   $V$ is injective if and only if $\theta(1)$ is central in $M$, and $\theta(N) = \theta(1)M$;
\item\label{thm:funcs_to_HMs:four}
   for quantum functions, surjective homomorphisms correspond to injective quantum functions.
\end{enumerate}
\end{theorem}
\begin{proof}
Proposition~\ref{prop:defn_V_theta} shows that $V^\theta$ is always coinjective.  Now let $V$ be an arbitrary coinjective quantum relation, and let $H_e = \overline{\lin}\{ v^*\xi : \xi\in K, v\in V \} \subseteq H$, and let $e$ be the projection of $H$ onto $H_e$.  As $V^*$ is a left $M'$-module, it follows that $H_e$ is an $M'$-invariant subspace, and so $e\in M$.  By corestriction, every element of $V^*$ is a map in $\mc B(K,H_e)$, say giving $U \subseteq \mc B(K, H_e)$ which is a left $N'$-module.  Set $M_e = eMe \subseteq \mc B(H_e)$ a von Neumann algebra, so by \cite[Proposition~II.3.10]{TakesakiI} we have that $(M_e)' = (M')_e = eM' \subseteq \mc B(H_e)$.  Thus $U$ is a right $(M_e)'$-module because $V$ is a right $M'$-module.  Then Theorem~\ref{thm:SG} (applied as in the proof of Corollary~\ref{corr:kornells_bijec}) shows that there is a unique normal unital $*$-homomorphism $\theta_e \colon N \to M_e$ with $V^{\theta_e *} = U$.  Let $\theta$ be $\theta_e$ composed with the inclusion $M_e \to M$, so in particular, $\theta(1) = e$.  Then $v\in V^{\theta*}$ if and only if $vy = \theta(y) v = \theta_e(y) ev$ for $y\in N$.  So firstly $v = v1 = ev$, and secondly we see that $v\in V^{\theta *}$ if and only if $ev \in U$.  So $V^* = V^{\theta *}$.

Conversely, starting with $\theta$, set $e=\theta(1)$.  Every $v\in V^{\theta *}$ satisfies that $v=v1=ev$, and Corollary~\ref{corr:density_cond} shows that $H_e = e(H) = \overline\lin V^{\theta *}K$.  By the uniqueness clause of Theorem~\ref{thm:SG}, if we construct $\theta_e$ from $V^{\theta *}$ as in the previous paragraph, then $\theta_e(y) = e \theta(y) e$ for each $y\in N$, and so $\theta$ is exactly $\theta_e$ composed with the inclusion $M_e\to M$.  This shows the claimed bijection.

To show \ref{thm:funcs_to_HMs:one}, we have to show that when $\theta$ is unital, $V^\theta$ is cosurjective (the other implication is already shown in Corollary~\ref{corr:kornells_bijec}), which by Lemma~\ref{lem:com_theta_image} is equivalent to $\theta(N)' \supseteq M'$, but this is clear as $\theta(N) \subseteq M$.

Let $z$ be the central projection with $\ker(\theta) = (1-z)N$ as in Lemma~\ref{lem:kernal_theta}, so $V\circ V^* = zN'$.  Then $\theta$ is injective if and only if $z=1$ if and only if $V\circ V^*=N'$ if and only if $V$ is surjective (as always $V\circ V^* \subseteq N'$ here).  This shows \ref{thm:funcs_to_HMs:two}.

By Lemma~\ref{lem:com_theta_image}, $V^*\circ V = \theta(1)\theta(N)'$.  If $V$ is injective, then $\theta(1) \theta(N)' \subseteq M'$, so $\theta(1) \in M'$ and hence $e=\theta(1)$ is central.  As $\theta(N) \subseteq M$, we have $M' \subseteq \theta(N)'$ and so $e\theta(N)' \subseteq M'$ means $eM' \subseteq e\theta(N)' = e^2\theta(N)' \subseteq eM'$ and so we have equality.  Inside $\mc B(e(H))$, again by \cite[Proposition~II.3.10]{TakesakiI}, we have $(eM)' = eM' = e\theta(N)' = (e\theta(N))'$ and so $eM = e\theta(N) = \theta(N)$, showing the ``only if'' clause of \ref{thm:funcs_to_HMs:three}.  Conversely, if $e$ is central and $\theta(N) = eM$ then $e\theta(N)' = (e\theta(N))' = (eM)' = eM' = eM'e \subseteq M'$ and so $V$ is injective.

Finally, if $V$ is quantum function, then $\theta$ is unital, and so \ref{thm:funcs_to_HMs:three} gives that $V$ is injective if and only if $\theta$ is surjective, showing \ref{thm:funcs_to_HMs:four}.
\end{proof}

Finally, we consider composition in either category; compare \cite[Section~4]{kornell2015quantumfunctions} and \cite[Section~7]{Kornell_QuantumSets}.

\begin{proposition}\label{prop:hm_to_qr_composition}
For normal $*$-homomorphisms $\theta_1 \colon M_1 \to M_2$ and $\theta_2 \colon M_2 \to M_3$, we have that $V^{\theta_1} \circ V^{\theta_2} = V^{\theta_2 \circ \theta_1}$.
\end{proposition}
\begin{proof}
One can show this directly, but the argument is tedious.  Instead, this will follow from the generalisation Proposition~\ref{prop:cp_to_qr_composition} below, taking account of Proposition~\ref{prop:rel_from_HM_is_from_CP}.
\end{proof}

When a quantum relation $V$ is both coinjective and injective, also $V^*$ has these properties, and so there is a normal $*$-homomorphism $\theta \colon N \to M$ with $V=V^\theta$, and a normal $*$-homomorphism $\theta^\star \colon M \to N$ with $V^{\theta^\star} = V^*$.  We now describe $\theta^\star$.  Let $z\in N\cap N'$ be the central projection with $\ker(\theta) = (1-z)N$, see Lemma~\ref{lem:kernal_theta}, and recall that $\theta(1)\in M$ is also a central projection, Theorem~\ref{thm:funcs_to_HMs}\ref{thm:funcs_to_HMs:three}, with $\theta(N) = \theta(1)M$.  Then $\theta$ restricts to a $*$-isomorphism $zN \to \theta(N) = \theta(1)M$.  Let $\theta^{-1}$ be the inverse, and define $\theta^\star(x) = \theta^{-1}(\theta(1)x) \in zN \subseteq N$ for each $x\in M$, so $\theta^\star$ is a normal $*$-homomorphism.

\begin{proposition}\label{prop:star_is_inverse}
Let $V$ be both injective and coinjective, and let $\theta^\star$ as above.  Then $V^* = V^{\theta^\star}$.
\end{proposition}
\begin{proof}
Starting with the definition, we have the following chain of implications:
\begin{align*}
v \in V^{\theta^\star} 
\quad&\iff\quad   xv = v \theta^\star(x)  = v \theta^{-1}(\theta(1)x)  \qquad (x\in M) \\
&\implies\quad   \theta(y) v = v y \qquad (y=yz\in N).
\end{align*}
Notice that $\theta^\star(1) = \theta^{-1}(\theta(1)) = z$ and so $v\in V^{\theta^\star}$ implies that $v = 1v = v \theta^\star(1) = vz$.  So for $y\in N$, as $\theta(y) = \theta(zy)$, we see that $\theta(y)v = \theta(zy)v = vzy = vy$ for any $y\in N$, that is, $v^* \in V^\theta = V$.

Conversely, given $v\in V^*$ we have that $\theta(y)v = vy$ for each $y\in N$.  In particular, $v = \theta(1)v$.  Given $x\in M$, set $y = \theta^{-1}(\theta(1)x) \in zN$, so $\theta(y) = x\theta(1)$, and hence $xv = x\theta(1)v = \theta(y) v = v y = v \theta^\star(x)$.  So $v\in V^{\theta^\star}$ and $V^* = V^{\theta^\star}$ as claimed.
\end{proof}

We finish by making some remarks about invertible relations, following \cite[Section~4]{Kornell_QuantumSets}.

\begin{proposition}
Let $U$ be a partial function which is invertible.  Then $U^*$ is the inverse of $U$, and so $U$ is a function which is both injective and surjective.
\end{proposition}
\begin{proof}
Let $U$ from $M$ to $N$ with inverse $V$ from $N$ to $M$.  This means that $V \circ U = M'$ and $U\circ V = N'$.
By hypothesis, $U \circ U^* \subseteq N'$, but as before, as $U\circ U^*$ is also a $N'$-bimodule, it forms a weak$^*$-closed ideal in $N'$ and hence equals $zN'$ for some central projection $z$.  Then $V^*\circ V \circ U \circ U^* = V^* \circ M' \circ U^* = V^*\circ U^* = (U\circ V)^* = N'$ but also $V^*\circ V \circ U \circ U^* = V^* \circ V \circ N'z = V^* \circ (Vz) \subseteq \mc B(K)z$ and so necessarily $z=1$ (as $N'$ is unital, for example).

So $U\circ U^* = N'$.  Then $U^*\circ U = V \circ U \circ U^* \circ U = V \circ N' \circ U = V\circ U = M'$.  So $U$ has inverse $U^*$ and so $U$ is a function, both injective and surjective.
\end{proof}

As remarked after \cite[Proposition~4.7]{Kornell_QuantumSets}, we can certainly have invertible relations which are not partial functions.

\section{Channels to relations}\label{sec:channels_to_relations}

We now study a notion similar to ``quantum confusability graphs''.  We begin with motivation from the classical situation. 

\begin{example}\label{eg:classical_channel}
A noisy communication channel from a (finite) set $X$ to $Y$ can be modelled as a conditional probability distribution $p(y|x)$.  So $p(y|x)\in [0,1]$ and $\sum_{y\in Y} p(y|x) = 1$ for each $x\in X$.  From this we obtain a UCP map $\theta \colon \ell^\infty(Y) \to \ell^\infty(X)$;
\[ \theta(f)(x) = \sum_{y\in Y} p(y|x) f(y) \qquad (x\in X, f\in \ell^\infty(Y)). \]
Let us say (in a perhaps non-standard way) that the \emph{relation} given by the channel, or by $\theta$, is $R = \{ (x,y) : p(y|x) > 0 \}$ that is, the pairs $(x,y)$ where, with non-zero probability, the symbol $y$ can be obtained by transmitting $x$.  The confusability graph of the channel is the relation on $X$ where $x_1 \sim x_2$ when there is some $y$ with both $(x_1,y)\in R$ and $(x_2,y)\in R$.
\end{example}

We now seek to do something similar for general von Neumann algebras.  Verdon looked at an abstract version of this idea in \cite{Verdon_CovQuantumCombs}, see Section~\ref{sec:CP_to_QR_again} below.  A concrete construction, for algebras of the form $L^\infty(\mc X)$, is given in \cite[Section~6]{kornell2026quantumgraphshomomorphisms}; we make some further comments below.

Let $M\subseteq\mc B(H)$ and $N\subseteq\mc B(K)$ be von Neumann algebras, and let $\theta \colon N \to M$ be a normal completely positive (CP) map.  Then $\theta$ admits a Stinespring dilation: there is a normal unital $*$-homomorphism $\pi_0\colon N \to \mc B(H_0)$, an operator $v_0\colon H \to H_0$, and a normal unital $*$-homomorphism $\rho\colon M' \to \pi_0(N)'$, such that
\[ \theta(y) = v_0^*\pi_0(y)v_0 \quad (y\in N), \qquad v_0x' = \rho(x')v_0 \quad (x'\in M'). \]
Notice that any such ``dilation'' $(\pi_0,v_0,H_0)$ gives a normal CP map $N\to\mc B(H)$, and the existence of $\rho$ ensures this maps into $M$.  We have that $\theta$ is unital (UCP) if and only if $v_0$ is an isometry.  We say that the dilation is \emph{minimal} if $\lin\{ \pi_0(y)v_0\xi : \xi\in H, y\in N\}$ is dense in $H_0$.  A minimal dilation always exists with a suitable $\rho$, \cite[Theorem~IV.5.5]{TakesakiI} for example, and is unique up to unitary equivalence.

Continuing, we can then dilate $\pi_0$ to obtain an isometry $u\colon H_0 \to K\otimes L$ with $u\pi_0(y) = (y\otimes 1)u$ for each $y\in N$, a technique used before.  Then $\theta(y) = v_0^* \pi_0(y) v_0 = v_0^* u^*(y\otimes 1) u v_0$, and so setting $v = uv_0$ we find that $\theta(y) = v^*(y\otimes 1)v$.  Now suppose we have some operator $v\colon H_0 \to K \otimes L$, for some $L$, with $\theta(y) = v^*(y\otimes 1)v$ for $y\in N$.  The map $\iota \colon \pi_0(y) v_0 \xi \mapsto (y\otimes 1)v\xi$ extends by linearity and continuity to an isometry $H_0 \to K\otimes L$, because we have two dilations of $\theta$, that is, because
\begin{align*}
( (y_1\otimes 1)v\xi_1 | (y_2\otimes 1)v\xi_2 )
&= ( \xi_1 | v^*(y_1^*y_2\otimes 1)v \xi_2 )
= ( \xi_1 | \theta(y_1^*y_2) \xi_2 )   \\
&= ( \pi_0(y_1)v_0\xi_1 | \pi_0(y_2)v_0\xi_2 )
\qquad\qquad (y_1,y_2\in N, \xi_1,\xi_2\in H).
\end{align*}
Then $\iota v_0 = v$ and $\iota\pi_0(y) = (y\otimes 1)\iota$ for each $y\in N$.  We can then define $\rho' \colon M' \to \mc B(K\otimes L)$ by $\rho'(x') = \iota \rho(x') \iota^*$.  As $\iota$ is an isometry, $\rho'$ is a $*$-homomorphism.  For $x'\in M', y\in N$ we see that $\rho'(x')(y\otimes 1) = \iota \rho(x') \pi_0(y) \iota^* = \iota \pi_0(y) \rho(x') \iota^* = (y\otimes 1)\rho'(x')$, and so $\rho'(x') \in (N\otimes 1)' = N'\vnten\mc B(L)$.  Furthermore, $\rho'(x') v = \iota \rho(x') \iota^*\iota v_0 = \iota v_0 x' = v x'$, and so $\rho'$ satisfies analogous properties to $\rho$.

\begin{theorem}\label{thm:UCP_to_QR}
Let $\theta \colon N \to M$ be a normal CP map with dilation given by an operator $v\colon H \to K\otimes L$, so that $\theta(y) = v^*(y\otimes 1)v$ for $y\in N$.  Then
\[ V = \lin\{ y' (1\otimes\langle\xi|) v : \xi\in L, y'\in N' \}\clos^{w^*} \subseteq\mc B(H,K) \]
is a quantum relation from $M$ to $N$, whose definition depends only on $\theta$ and not upon the dilation chosen.  Indeed, let $(\pi_0, H_0, v_0)$ be a minimal dilation of $\theta$, and consider the quantum function $V^{\pi_0} \subseteq \mc B(H_0, K)$.  Then $V$ is the weak$^*$-closure of $\{ x v_0 : x\in V^{\pi_0}\}$.
\end{theorem}
\begin{proof}
As above, there is $\rho' \colon M' \to N'\vnten\mc B(L)$ which is intertwining for $v$.  For $\xi\in H_1$ we have $(1\otimes\langle\xi|)v x' = (1\otimes\langle\xi|)\rho'(x') v \in (1\otimes\langle\xi|) (N'\vnten\mc B(L)) v \in V$ by weak$^*$-continuity.  So $V$ is a right $M'$-module, and left $N'$-module by definition, so a quantum relation from $M$ to $N$.

Let $(\pi_0, H_0, v_0)$ be a minimal dilation of $\theta$, and as above, consider the isometry $\iota \colon H_0 \to K\otimes L$, which satisfies that $\iota v_0 = v$.  Setting
\[ W_0 = \lin\{ y' (1\otimes\langle\xi|) \iota : \xi\in L, y'\in N' \} \subseteq \mc B(H_0, K), \]
and $W$ to be the weak$^*$-closure of $W_0$, we have that $V$ is the weak$^*$-closure of $\{ x v_0 : x\in W_0 \}$.  As the operation of pre-composing with $v_0$ is weak$^*$-continuous, also $V$ is the weak$^*$-closure of $\{ x v_0 : x\in W \}$.  We shall show that $W = V^{\pi_0}$.

Clearly $W$ is a $N'$-$\mc B(H_0)'$-bimodule, and so a quantum relation from $\mc B(H_0)$ to $N$.  
A typical element of $W_0\circ W_0^*$ is $y_1'(1\otimes\langle\xi_1|)\iota \iota^* (1\otimes |\xi_2\rangle) y_2' = y_1' (\id\otimes\omega_{\xi_1,\xi_2})(\iota\iota^*) y_2' \in N'$ because $\iota\iota^* \in (N\otimes 1)' = N'\vnten\mc B(L)$ by the commutation relation $(y\otimes 1)\iota = \iota\pi_0(y)$ for $y \in N$.  Taking weak$^*$-closures shows that also $W\circ W^* \subseteq N'$, that is, $W$ is coinjective.  By Theorem~\ref{thm:funcs_to_HMs}, there is a normal $*$-homomorphism $\pi \colon N \to \mc B(H_0)$ with $W = V^\pi$.  In particular, $x \pi(y) = y x$ for $y\in N, x\in W$.  However, for $x =  y' (1\otimes\langle\xi|) \iota \in W_0$ we find that
\[  y' (1\otimes\langle\xi|) \iota \pi(y) = y y' (1\otimes\langle\xi|) \iota  =  y' (1\otimes\langle\xi|) (y\otimes 1)\iota =  y' (1\otimes\langle\xi|) \iota \pi_0(y)
\qquad (y\in N). \]
As this holds for all $y',\xi$ we obtain $\iota\pi(y)=\iota\pi_0(y)$, and as $\iota$ is an isometry, we see that $\pi=\pi_0$.  Hence, as claimed,
\[ W = V^{\pi_0} = \{ x\in\mc B(H_0, K) : yx = x\pi_0(y) \ (y\in N) \}. \]

We have hence shown the second claim of the theorem.  But then we have expressed $V$ just using data from the minimal dilation of $\theta$, and hence it follows that the initial definition of $V$ does not depend upon the choice of $v$.
\end{proof}

\begin{definition}
Write $V^\theta$ for the quantum relation defined by a CP map $\theta$.
\end{definition}

\begin{example}\label{eg:classical_channel_relation}
Let's return to Example~\ref{eg:classical_channel}, that of a (classical) channel from $X$ to $Y$.  A dilation of the UCP map $\theta \colon \ell^\infty(Y) \to \ell^\infty(X)$ is given by $v\colon \ell^2(X) \to \ell^2(Y) \otimes \ell^2(X); \delta_x \mapsto \sum_y p(y|x)^{1/2} \delta_y \otimes \delta_x$ as then
\[ v^*(f\otimes 1)v\delta_x = v^* \sum_y p(y|x)^{1/2} f(y) \delta_y \otimes \delta_x
= \sum_y p(y|x_ f(y) \delta_x = \theta(f) \delta_x. \]
Let $e_y\in \ell^\infty(Y)$ be the minimal projection onto the single point $y$, and let $e_{y,x}$ be the rank-one map $\ell^2(X) \to \ell^2(Y); \delta_x\mapsto\delta_y$ and $\delta_z\mapsto 0$ for $z\not=x$.  We see that $(1\otimes\langle\delta_x|)v \delta_z = \delta_{x,z} \sum_y p(y|x)^{1/2} \delta_y = \sum_y p(y|x)^{1/2} e_{y,x}\delta_z$.  Hence the quantum relation is
\begin{align*}
V &= \lin\{ e_y (1\otimes\langle\delta_x|)v : x\in X, y\in Y\}\clos^{w^*}
= \lin\{ p(y|x)^{1/2} e_{y,x} : x\in X, y\in Y\}\clos^{w^*} \\
&= \lin\{ e_{y,x} : p(y|x)\not=0 \}\clos^{w^*}.
\end{align*}
That is, the quantum relation given by the relation $R$ associated to the channel, as we might hope; compare Example~\ref{eg:adj_in_comm_case} below.
\end{example}

\begin{remark}\label{rem:graph_from_qr}
Consider $V^{\theta *}\circ V^\theta$, a quantum relation on $M$.  This equals
\begin{align*}
& \lin\{ v^* (1\otimes|\xi_1\rangle)y'(1\otimes\langle\xi_2|)v : y'\in N', \xi_1,\xi_2\in H_1 \}\clos^{w^*} \\
&= \lin\{ v^* (y' \otimes T)) v : y'\in N', T\in\mc B(H_1) \}\clos^{w^*} \\
&= \lin\{ v^* S v : S \in N' \vnten \mc B(H_1) \}\clos^{w^*}
\end{align*}
Comparing with \cite[Remark~7.6]{daws_quantum_graphs} we see that $V^{\theta *}\circ V^\theta$ is the pullback of the trivial quantum graph $N'$, equivalently, is the quantum confusability graph of $\theta$.  Hence $V^\theta$ can indeed be regarded as the ``relation'' associated to $\theta$.
\end{remark}

\begin{remark}\label{rem:UCP_to_cosurjective}
In particular, when $\theta$ is UCP, we see that $1\in V^{\theta *}\circ V^\theta$ so $M' \subseteq V^{\theta *}\circ V^\theta$ by the bimodule property, and so $V^\theta$ is cosurjective.
\end{remark}

We now show that our construction here generalises the construction in the previous section, for Quantum Functions and $*$-homomorphisms, and so $V^\theta$ is not confusing notation.

\begin{proposition}\label{prop:rel_from_HM_is_from_CP}
Let $\theta \colon N \to M$ be a normal $*$-homomorphism.  Then the quantum relation $V$ given by Theorem~\ref{thm:UCP_to_QR} is $V^\theta$, the quantum function associated to $\theta$.
\end{proposition}
\begin{proof}
Considering $\theta$ as a CP map, we find a dilation $(\pi_0, v_0, H_0)$ as follows.  Set $e=\theta(1)$ a projection in $M$ with $e \theta(y) = \theta(y) e = \theta(y)$ for each $y\in N$.  Let $H_0 = e(H)$ and $\pi_0$ be $\theta$ considered with codomain $eMe \subseteq \mc B(H_0)$.  Let $v_0 \colon H \to H_0$ be the projection $e$.  Then
\[ v_0^* \pi_0(y) v_0 = e \theta(y) e = \theta(y) \qquad (y\in N), \]
and so we do have a dilation.  Also $\lin\{ \pi_0(y) v_0 \xi : \xi\in H, y\in N\} = \lin\{ \theta(y)e\xi : \xi\in H, y\in N \} = e(H) = H_0$ so the dilation is minimal.  
By Theorem~\ref{thm:UCP_to_QR}, it follows that $V = V^{\pi_0}e$, as this is weak$^*$-closed because $e$ is idempotent.

Let $x\in V^{\pi_0}$, so $x\colon H_0 \to K$ satisfies $yx = x\pi_0(y) = x\theta(y)|_{H_0}$ for $y\in N$.  Thus $yxe = x\theta(y)e = x\theta(y)$ and so $xe\in V^\theta$.  Conversely, let $x \in V^\theta$ and set $x_0 = x|_{H_0} \colon H_0 \to K$.  Then $yx_0 = yx|_{H_0} = x\theta(y)|_{H_0} = x|_{H_0}\pi_0(y) = x_0\pi_0(y)$ for $y\in N$, and so $x_0 \in V^{\pi_0}$, and then $x_0e = xe = x\theta(1)=1x=x$.  It follows that $V = V^{\pi_0}e = V^\theta$ as claimed.
\end{proof}

So far, the definition of $V^\theta$, while being independent of the dilation of $\theta$ chosen, does depend upon the representations $M\subseteq \mc B(H)$ and $N\subseteq\mc B(K)$.  We now show that if we vary $H,K$ then $V^\theta$ varies in exactly the way we expect from Corollary~\ref{corr:indep_rep}.

\begin{proposition}
Given a UCP map $\theta \colon N \to M$, the definition of $V^\theta$ does not depend upon the choice of Hilbert spaces which $M,N$ act on.
\end{proposition}
\begin{proof}
Let $\theta_M \colon M \to \mc B(H_1)$ be a unital injective normal $*$-homomorphism dilated by $u_M \colon H_1 \to H\otimes L_1$, and similarly for $N$.  Let $\theta_1 \colon N \to M$ be the CP map $\theta$, considered now as a map from $\theta_N(N)\subseteq\mc B(K_1)$ to $\theta_M(M)\subseteq\mc B(H_1)$, that is, $\theta_M\circ\theta = \theta_1\circ\theta_N$.  Let $(\pi_0', H_0', v_0')$ be a minimal dilation of $\theta_1$, and again let $(\pi_0, H_0, v_0)$ be a minimal dilation of $\theta$.  For $y_1,y_2\in N, \xi_1,\xi_2\in H_1$, we have
\begin{align*}
( \pi_0'\theta_N(y_1) v_0' \xi_1 | \pi_0'\theta_N(y_2) v_0' \xi_2 )
&= ( \xi_1 | \theta_M(\theta(y_1^*y_2)) \xi_2 )
= ( u_M\xi_1 | (\theta(y_1^* y_2)\otimes 1) u_M\xi_2 )   \\
&= \big(  (\pi_0(y_1)v_0 \otimes 1) u_M\xi_1 \big| (\pi_0(y_2)v_0 \otimes 1) u_M\xi_2 \big).
\end{align*}
So the map $\alpha \colon \pi_0'\theta_N(y)v_0'\xi \mapsto (\pi_0(y)v_0 \otimes 1)u_M\xi$ extends by linearity and continuity to an isometry $H_0' \to H_0\otimes L_1$, which satisfies $\alpha \pi_0'\theta_N(y) = (\pi_0(y)\otimes 1)\alpha$ and $\alpha v_0' = (v_0\otimes 1)u_M$.  Consequently, $\alpha\alpha^* \in (\pi_0(N)\otimes 1)'$, and $\pi_0'\theta_N(y) = \alpha^*(\pi_0(y)\otimes 1)\alpha$.

By definition, $x\colon H_0' \to K_1$ is in $V^{\pi_0'}$ if and only if $\theta_N(y) x = x \pi_0'(y)$ for each $y\in N$, and then arguing as in the proof of Proposition~\ref{prop:changing_base_space}, this is equivalent to $u_N x \alpha^* \in V^{\pi_0} \vnten \mc B(L_1,L_2)$.  So $V^{\theta_1}$ is the weak$^*$-closure of maps of the form $x v_0'$ for such $x$.  As $v_0' = \alpha^*(v_0\otimes 1) u_M$, we see that $x v_0' =  x \alpha^*(v_0\otimes 1) u_m = u_N^* x_0 (v_0\otimes 1) u_M$ for some $x_0 \in V^{\pi_0} \vnten \mc B(L_1,L_2)$.  Then $x_0 (v_0\otimes 1) \in V^{\theta} \vnten \mc B(L_1,L_2)$, and so $xv_0' \in u_N^*(V^{\theta} \vnten \mc B(L_1,L_2))u_M$ which by Remark~\ref{rem:changing_base_spaces} is exactly $\tilde V^\theta$, the quantum relation from $\theta_M(M)$ to $\theta_N(N)$ induced by $V^\theta$.

Conversely, a weak$^*$-dense subspace of $\tilde V^\theta$ is $\{ u_N^*(xv_0 \otimes t) u_M : x\in V^{\pi_0}, t\in\mc B(L_1, L_2) \}$.  Given such an $a = u_N^*(xv_0 \otimes t) u_M \in \tilde V^\theta$, set $x_1 = u_N^*(x\otimes t) \alpha$, so $u_N x_1 \alpha^* = u_N u_N^* (x\otimes t) \alpha\alpha^* \in V^{\pi_0} \vnten \mc B(L_1, L_2)$ by the bimodule property.  Hence $x_1 \in V^{\pi_0'}$, and so $x_1 v_0' \in V^{\theta_1}$.  Then $x_1 v_0' = u_N^* (x\otimes t) \alpha v_0' = u_N^* (xv_0\otimes t) u_M = a$.  As $V^{\theta_1}$ is weak$^*$-closed, this shows that $\tilde V^\theta \subseteq V^{\theta_1}$, and the proof is complete.
\end{proof}

We also have an exact analogue of Proposition~\ref{prop:hm_to_qr_composition}.

\begin{proposition}\label{prop:cp_to_qr_composition}
Let $\theta_1 \colon M_1 \to M_2$ and $\theta_2 \colon M_2 \to M_3$ be CP maps, where $M_i \subseteq\mc B(H_i)$ are von Neumann algebras for $i=1,2,3$.  Then $V^{\theta_2 \circ \theta_1} = V^{\theta_1} \circ V^{\theta_2}$.
\end{proposition}
\begin{proof}
Let $u_1 \colon H_2 \to H_1 \otimes L_1$ dilate $\theta_1$, and $u_2 \colon H_3 \to H_2 \otimes L_2$ dilate $\theta_2$.  For $x\in M_1$ we have
\[ \theta_2(\theta_1(x)) = u_2^*(\theta_1(x)\otimes 1)u_2 = u_2^*(u_1\otimes 1)^*(x\otimes 1\otimes 1)(u_1\otimes 1) u_2, \]
and so $u = (u_1\otimes 1) u_2$ dilates $\theta_2 \circ \theta_1$.  Hence, as we take linear spans and closures,
\begin{align*}
V^{\theta_2 \circ \theta_1}
&= \lin \{ y' (1\otimes \langle \xi\otimes\eta|)u : \xi\in L_1, \eta\in L_2, y'\in M_1' \}\clos^{w^*}   \\
&= \lin \{ y' (1\otimes \langle \xi|)u_1 (1\otimes \langle\eta|)u_2 : \xi\in L_1, \eta\in L_2, y'\in M_1' \}\clos^{w^*}  \\
&= \lin\{ x (1\otimes \langle\eta|)u_2 : x\in V^{\theta_1}, \eta\in L_2 \}\clos^{w^*}  \\
&= \lin\{ x y' (1\otimes \langle\eta|)u_2 : x\in V^{\theta_1}, \eta\in L_2, y'\in M_2' \}\clos^{w^*}  \\
&= V^{\theta_1} \circ V^{\theta_2},
\end{align*}
in the penultimate step using that $V^{\theta_1}$ is a right $M_2'$-module.
\end{proof}

This proposition shows that $\theta \mapsto V^\theta$ is a functor between the relevant categories.  See also our discussion of \cite{Verdon_CovQuantumCombs} below, in particular Propositions~\ref{prop:functor_A_to_V} and~\ref{prop:A_vs_theta}.

Given a CP map $\theta \colon N \to M$ with dilation $\theta(y) = v^*(y\otimes 1)v$ for some operator $v\colon H \to K\otimes L$, take an orthonormal basis $(e_i)_{i\in I}$ for $L$, so there are operators $b_i \colon H \to K$ with $v(\xi) = \sum_{i\in I} b_i(\xi)\otimes e_i$ for each $\xi\in H$, and hence $\sum_i b_i^* b_i = v^*v < \infty$.  Then $\theta(y) = \sum_{i\in I} b_i^* y b_i$ is a \emph{Kraus representation} of $\theta$.  The following should be compared to \cite[Proposition~6.1(3)]{kornell2026quantumgraphshomomorphisms}.

\begin{proposition}\label{prop:kraus_to_V}
Let $\theta \colon N \to M$ have Kraus representation $\theta(y) = \sum_{i\in I} b_i^* y b_i$ for $y\in N$.  The associated quantum relation is
\[ V^\theta = \lin\{ y' b_i : y'\in N', i\in I \}\clos^{w^*}. \]
\end{proposition}
\begin{proof}
As any $\xi\in L$ has the form $\xi = \sum_i \xi_i e_i$ for some square-summable $(\xi_i)$, we see that $(1\otimes \langle\xi|)v = \sum_i \overline{\xi_i} b_i$, and so the claim follows from the definition of $V^\theta$ from Theorem~\ref{thm:UCP_to_QR} and a standard approximation argument.
\end{proof}

\begin{remark}\label{rem:not_all_cosurj}
As in Remark~\ref{rem:UCP_to_cosurjective}, when $\theta$ is UCP, $V^\theta$ is cosurjective.  We might wonder if every cosurjective quantum relation $V$ from $M$ to $N$ arises from a UCP map $N\to M$?  Here is an example to show that this is not the case.

Let $M=N=\mathbb M_2$ so a quantum relation from $M$ to $N$ is simply a subspace $V\subseteq \mc B(\mathbb C^2, \mathbb C^2) = \mathbb M_2$.  Set
\[ \hat u_1 = \begin{pmatrix} 3 & 0 \\ 0 & 2 \end{pmatrix}, \quad
u_1 = \frac{1}{\sqrt 2}\begin{pmatrix} 1 & 1 \\ -1 & 1 \end{pmatrix} \hat u_1, \quad
u_2 = \begin{pmatrix} \sqrt 8 & 0 \\ 0 & \sqrt 3 \end{pmatrix}. \]
Then $u_1^*u_1 = \hat u_1^* \hat u_1$ and $u_1^*u_1 - u_2^*u_2 = 1$.  Set $V = \lin \{ u_1, u_2 \}$ so $1 \in V^*\circ V$ and so $V$ is cosurjective.

Suppose there is a  UCP $\theta \colon N \to M$ giving $V$.  Let $(\pi,u,H)$ be a minimal Stinespring representation of $\theta$, so $\pi \colon \mathbb M_2 \to \mc B(H)$ and $u\colon \mathbb C^2\to H$ is an isometry with $u^*\pi(y)u = \theta(y)$ for each $y\in\mathbb M_2$.  Then $H = \mathbb C^2 \otimes H_1$ for some $H_1$ and $\pi(y) = y\otimes 1$.  Pick an orthonormal basis $(\delta_i)$ for $H_1$ and let $b_i\in\mathbb M_2$ with $u(\xi) = \sum_i b_i(\xi)\otimes\delta_i$ for each $\xi\in\mathbb C^2$.  Then $\theta(y) = \sum_i b_i^* y b_i$ is a Kraus representation.  Minimality means that $\lin\{\pi(y) u\xi : y\in\mathbb M_2, \xi\in\mathbb C^2\}$ is dense in $H$, that is,
\[ \mathbb C^2 \otimes H_1
= \overline\lin\big\{ \sum_i yb_i\xi \otimes \delta_i : y\in\mathbb M_2, \xi\in\mathbb C^2 \big\}
= \lin\big\{ \sum_i e_{st} b_i\delta_r \otimes \delta_i : 1 \leq s,t,r\leq 2 \big\}, \]
where we remove the closure as the 2nd expression is clearly finite-dimensional.  Hence $H_1$ is finite-dimensional (in fact, of dimension at most 4).  Indeed, as $e_{st} b_i\delta_r = (\delta_t|b_i\delta_r) \delta_s$ this is equivalent to
\[ H_1 = \lin \{ \xi_{t,r} : 1\leq t,r\leq 2 \}
\quad\text{ where }\quad
\xi_{t,r} = \sum_i (\delta_t|b_i\delta_r)\delta_i. \]
Suppose we have scalars $s_i$ with $\sum_i s_i b_i = 0$.  Then for each $t,r$ we have
\[ 0 = \sum_i (\delta_t|s_ib_i\delta_r)
= (\overline s | \xi_{t,r}), \]
where $\overline s = \sum_i \overline{s_i} \delta_i$.  As $\{\xi_{t,r}\}$ spans $H_1$ this is equivalent to $\overline{s}=0$, that is, $s_i=0$ for each $i$.  So the $b_i$ are linearly independent.  By Proposition~\ref{prop:kraus_to_V} we have that $V=\lin\{b_i\}$ and so the $b_i$ form a basis for $V$.

So $\{ u_1, u_2\}$ and $\{b_1,b_2\}$ are bases for $V$.  There is hence a matrix $B\in\mathbb M_2$ with
\begin{align*}
B \begin{pmatrix} u_1 \\ u_2 \end{pmatrix} 
= \begin{pmatrix} b_1 \\ b_2 \end{pmatrix}
\implies
1 &= b_1^*b_1 + b_2^*b_2 = \begin{pmatrix} u_1^* &  u_2^* \end{pmatrix} B^*B\begin{pmatrix} u_1 \\ u_2 \end{pmatrix} 
= \begin{pmatrix} u_1^* &  u_2^* \end{pmatrix} \begin{pmatrix} s & \mu \\ \overline\mu & t \end{pmatrix}
\begin{pmatrix} u_1 \\ u_2 \end{pmatrix}  \\
&= s u_1^*u_1 + \mu u_1^*u_2 + \overline\mu u_2^*u_1 + t u_2^*u_2,
\end{align*}
say, where as $B^*B$ is positive definite, $s,t\geq 0$ and $|\mu^2|\leq st$.  Looking at the top-right entry of $\mathbb M_2$, we have
\[ u_1 = \frac{1}{\sqrt 2} \begin{pmatrix} 3 & 2 \\ -3 & 2 \end{pmatrix}
\implies u_2^*u_1 = \frac{1}{\sqrt 2} \begin{pmatrix} 3\sqrt 8 & 2\sqrt 8 \\ -3\sqrt 3 & 2\sqrt 3 \end{pmatrix}
\implies
0 = - \mu 3\sqrt 3 + \overline\mu 2\sqrt 8, \]
but then $|\mu|3\sqrt 3 = |\mu| 2\sqrt 8$ contradiction.

This contradiction shows that $V$ cannot arise from a UCP map $\theta$.
\end{remark}

\begin{remark}\label{rem:all_qr_from_CP}
When $M = \mathbb M_m, N=\mathbb M_n$, as the commutants are trivial, by taking a basis $(b_i)$ of any $V\subseteq\mc B(\mathbb C^m, \mathbb C^n)$ and then using the $(b_i)$ as Kraus operators, we easily find a CP $\theta \colon N\to M$ with $V^\theta=V$, by Proposition~\ref{prop:kraus_to_V}.
We generalise this all finite-dimensional von Neumann algebras in Corollary~\ref{corr:all_fd_qrs_from_CP} below, using different techniques.
\end{remark}

It would be interesting to describe the quantum relations which can arise as $V^\theta$ with $\theta$ UCP.  What seems to provide the counter-example given in Remark~\ref{rem:not_all_cosurj} is that we construct $V$ with $V^*\circ V$ unital, but such that there is no spanning set $(b_i)$ of $V$ with $\sum_i b_i^*b_i = 1$.  That the same phenomena does not occur for quantum graphs, \cite[Lemma~2]{duan2009superactivationzeroerrorcapacitynoisy}, seems to be a slightly subtle interplay between using that a quantum graph is both self-adjoint, and unital, and so has a good notion of positivity.

\subsection{Pullbacks}\label{sec:pullbacks}

In \cite{Weaver_QuantumGraphs} Weaver defined the notion of a \emph{pullback}, a way to transform a quantum relation (over a single algebra) using a (U)CP map.  We now generalise this to quantum relations over two algebras, and show links with the previous section.  We first recall our treatment of pullbacks, see \cite[Theorem~7.5]{daws_quantum_graphs}.  Let $V\subseteq\mc B(H_1)$ be a quantum relation over a von Neumann algebra $M_1\subseteq\mc B(H_1)$, and let $\theta\colon M_1\to M_2$ be a normal CP map.  With $(\pi,L,v)$ a dilation of $\theta$, the pullback of $V$ is $\overleftarrow{V}^\theta$ the weak$^*$-closure of $v^*V_\pi v$, where $V_\pi$ is as in Lemma~\ref{lem:HM_mapping_preserves_qr}.  This definition is independent of the dilation chosen, and $\overleftarrow{V}^\theta \subseteq \mc B(H_2)$ is a quantum relation over $M_2$.  We write $\overleftarrow{V}$ when the CP map $\theta$ used is clear from context.

This definition readily generalises to quantum relations as studied in this paper.

\begin{proposition}\label{prop:pullback_defn}
For $i=1,2$, let $M_i\subseteq\mc B(H_i), N_i\subseteq\mc B(K_i)$ be von Neumann algebras, let $V \subseteq \mc B(H_1,K_1)$ be a quantum relation from $M_1$ to $N_1$, and let $\theta_M\colon M_1\to M_2$ and $\theta_N\colon N_1\to N_2$ be normal CP maps.  Treating $V \subseteq \mc B(H_1\oplus K_1)$ as a quantum relation over $M_1\oplus N_1$, let $\theta = \theta_N\oplus\theta_M$, and form the pullback $U = \overleftarrow{V}^\theta$.  Then:
\begin{enumerate}
    \item $1_{N_2} U 1_{M_2} = U$, so $U$ defines a quantum relation from $M_2$ to $N_2$;
    \item with $\theta_M(x) = u_M^*(x\otimes 1)u_M, \theta_N(y) = u_N^*(y\otimes 1)u_N$ dilations with, say, $u_M\colon H_2 \to H_1\otimes L_M, u_N\colon K_2\to K_1\otimes L_N$, as a subspace of $\mc B(H_2, K_2)$ we have that $U$ is the weak$^*$-closure of $u_N^* (V \vnten \mc B(L_M, L_N)) u_M$.
\end{enumerate}
\end{proposition}
\begin{proof}
We use a similar argument to Remark~\ref{rem:changing_base_spaces}.  Given the dilations of $\theta_M$ and $\theta_N$, we see that $\theta$ has a dilation given by
\[ u \colon H_2\oplus K_2 \to (H_1\oplus K_1) \otimes (L_M\oplus L_N) \cong (H_1 \otimes L_M) \oplus (H_1 \otimes L_N) \oplus (K_1 \otimes L_M) \oplus (K_1 \otimes L_N),  \]
namely $u = u_N \oplus 0\oplus 0\oplus u_M$.  Then by definition $U = u^*(V \vnten \mc B(L_M\oplus L_N)) u$.  Notice that $u 1_{M_2} = (1_{M_1}\otimes 1)u$ and similarly $u 1_{N_2} = (1_{N_1}\otimes 1)u$, and so $U 1_{M_2} = u^*(V\vnten\mc B(L_M\oplus L_N))(1_{M_1}\otimes 1)u = U$ because $V 1_{M_1} = V$, and similarly $1_{N_2} U = U$.
The final claim follows from the calculation in Remark~\ref{rem:changing_base_spaces}, where now as $u_M, u_N$ are general operators (and don't satisfy, for example, that $u_Mu_M^* \in (M\otimes 1)'$) we need to take the weak$^*$-closure.
\end{proof}

As for quantum confusability graphs, \cite[Remark~7.6]{daws_quantum_graphs}, we find that $V^\theta$, for $\theta$ a CP map, can be expressed as a pullback.

\begin{proposition}\label{prop:Vtheta_from_pullback}
Given a UCP map $\theta \colon N \to M$, the quantum relation $V^\theta$ is the pullback of the trivial relation $N'$ along the CP map $\theta \oplus 1 \colon N\oplus N \to M\oplus N$.
\end{proposition}
\begin{proof}
Let $v \colon H\to K\otimes L$ dilate $\theta$, so by definition, Theorem~\ref{thm:UCP_to_QR}, $V^\theta = \lin \{ y'(1\otimes \langle\xi|)v : \xi\in L, y'\in N'\}\clos^{w^*}$.  The CP map $\theta \oplus 1$ is dilated by the pair $(v,1)$, and so
\[ \overleftarrow{N'} = 1 (N'\vnten \mc B(L, \mathbb C)) v \clos^{w^*}
= \lin\{ y' (1\otimes \langle\xi|)v : \xi\in L, y'\in N' \}\clos^{w^*} = V^\theta, \]
as required.
\end{proof}

We next show that the pullback can be expressed using composition of relations, and the relations defined by the CP maps $\theta_M$ and $\theta_N$; for algebras of the form $L^\infty(\mc X)$, see also \cite[Proposition~6.1]{kornell2026quantumgraphshomomorphisms}.

\begin{proposition}\label{prop:pullback_as_qr_composition}
We have that $\overleftarrow{V} = V^{\theta_N *} \circ V \circ V^{\theta_M}$.
\end{proposition}
\begin{proof}
We continue as with the notation in Proposition~\ref{prop:pullback_defn}.  Then, as we take the weak$^*$-closure, it suffices to consider rank-one operators in $\mc B(L_M, L_N)$ and so
\begin{align*}
\overleftarrow{V} &= \lin \{ u_N^* (v \otimes t) u_M : v\in V, t\in\mc B(L_M, L_N) \}\clos^{w^*}  \\
&= \lin \{ u_N^* (v \otimes \rankone{\eta}{\xi}) u_M : v\in V, \xi\in L_M, \eta\in L_N, \}\clos^{w^*}  \\
&= \lin \{ u_N^*(1 \otimes |\eta\rangle) v (1\otimes \langle\xi|) u_M  : v\in V, \xi\in L_M, \eta\in L_N, \}\clos^{w^*} \\
&= \lin \{ u_N^*(1 \otimes |\eta\rangle) y' v x' (1\otimes \langle\xi|) u_M  : v\in V, \xi\in L_M, \eta\in L_N, x' \in M_1', y'\in N_1'\}\clos^{w^*},
\end{align*}
in the final step using that $V$ is an $N_1'$-$M_1'$-bimodule.  From Theorem~\ref{thm:UCP_to_QR}, we recognise the relations $V^{\theta_M}$ and $V^{\theta_N *}$, and so see that $\overleftarrow{V} = V^{\theta_N *} \circ V \circ V^{\theta_M}$.
\end{proof}

This gives a simple proof of how composition interacts with the pullback construction.

\begin{corollary}\label{corr:composition_pullbacks}
Given CP maps $\theta_1 \colon M_1\oplus N_1 \to M_2\oplus N_2$ and $\theta_2 \colon M_2\oplus N_2 \to M_3\oplus N_3$ of the form considered here, and a quantum relation $V$ from $M_1$ to $N_1$, we have that $\overleftarrow{V}^{\theta_2\circ\theta_1} = (V^{\leftarrow\theta_1})^{\leftarrow\theta_2}$.
\end{corollary}
\begin{proof}
This now follows immediately by using Proposition~\ref{prop:cp_to_qr_composition}.
\end{proof}

\begin{remark}
In particular, given a quantum relation $V$ over a single algebra $M_1$, and a single CP map $\theta \colon M_1\to M_2$, we have that $\overleftarrow{V} = V^{\theta *} \circ V \circ V^\theta$.  As in Remark~\ref{rem:graph_from_qr}, the quantum confusability graph of a UCP map $\theta$, being the pullback of the trivial quantum relation, is $V^{\theta *}\circ M_2'\circ V^\theta = V^{\theta *}\circ V^\theta$.
\end{remark}

\begin{remark}
Suppose that $\theta_M, \theta_N$ and actually unital normal $*$-homomorphisms.  Comparing Proposition~\ref{prop:pullback_defn} and Remark~\ref{rem:changing_base_spaces} shows immediately that $V_\theta$, the quantum relation induced by $V$ along the homomorphism $\theta$, is equal to $\overleftarrow{V}$, the pullback.  By Proposition~\ref{prop:pullback_as_qr_composition}, this equals $V^{\theta_N *} \circ V \circ V^{\theta_M}$, where $V^{\theta_M}$ and $V^{\theta_N}$ are here coinjective, as in Section~\ref{sec:quantum_functions}.  In particular, taking $\theta_M, \theta_N$ to be injective, so by Theorem~\ref{thm:funcs_to_HMs}, $V^{\theta_N}, V^{\theta_M}$ are quantum functions, this shows how we can compute the effect on quantum relations of changing the Hilbert space which our algebras act on, just using composition within the category of quantum relations.
\end{remark}

\section{Adjacency operators}\label{sec:adj_ops}

In the general theory of Quantum Graphs, at least when our von Neumann algebra $M$ is finite-dimensional, there is a bijection between quantum relations $V \subseteq \mc B(H)$ (that is, $M'$-bimodules, where $M\subseteq\mc B(H)$) and \emph{Quantum adjacency operators}, namely certain operators $A\in\mc B(L^2(M))$.  We explored this in \cite[Section~5]{daws_quantum_graphs} in a quite general way.  Starting with \cite{CW_RandomQGraphs, matsuda_class_m2} an alternative (or perhaps ``parallel'') approach is given by looking at maps $A \colon M \to M$ which are completely positive: this is further expounded in \cite{Wasilewski_Quantum_Cayley}.  Using a generalisation of the Choi Matrix construction, \cite[Section~5.4]{daws_quantum_graphs} and \cite[Section~3.1]{Wasilewski_Quantum_Cayley}, such $A$ biject with positive elements $e\in M \otimes M^\op$, and further, $A$ is Schur-idempotent if and only if $e$ is idempotent (and hence a projection).  We shall follow the conventions laid out in \cite{Wasilewski_Quantum_Cayley} and \cite{daws2025quantumgraphsinfinitedimensionshilbertschmidts}, which are explained more below.

In this section, we make the appropriate modifications to this theory to deal with general quantum relations between two von Neumann algebras $M$ and $N$, taking the approach suggested by Definition~\ref{defn:qrel_2}.  Choose faithful positive functionals $\varphi_M, \varphi_N$ on $M,N$ respectively, and represent $M$ on the GNS space $L^2(M) = L^2(\varphi_M)$, and $N$ on $L^2(\varphi_N)$.  We write $\Lambda = \Lambda_M\colon M \to L^2(M)$ for the GNS map, which in this case is a linear bijection, as $M$ is finite-dimensional.  We shall tend to suppress subscripts when it is clear to do so.

The multiplication map $M\otimes M \to M$ induces a (bounded) linear map $m=m_M \colon L^2(M) \otimes L^2(M) \to L^2(M)$; similarly for $N$.  In particular, we have the adjoint $m^* \colon L^2(M) \to L^2(M)\otimes L^2(M)$ which again can be regarded as a linear map $M\to M\otimes M$.  The \emph{Schur product} on $\mc B(L^2(M))$ is $A \star B = m(A\otimes B)m^*$ for $A,B\in\mc B(L^2(M))$.  As $M$ and $N$ are finite-dimensional, we shall identify $\mc B(L^2(M))$ with $\mc B(M)$, and so forth, and so also regard the Schur product as on $\mc B(M)$.
Motivated by Definition~\ref{defn:qrel_2}, we have the following.

\begin{definition}
A \emph{Quantum adjacency operator} from $M$ to $N$ is a completely positive (CP) map $A \colon M\oplus N \to M\oplus N$ which is Schur idempotent, and which satisfies $1_N A 1_M = A$.
\end{definition}

Again $1_N$ is the unit of $N \subseteq M\oplus N$, which here we identify also with the projection on $M\oplus N; (x,y) \mapsto (0,y)$, and similarly for $1_M$.
We remark that because (for example) $M\oplus N$ acts ``diagonally'' on $L^2(M) \oplus L^2(N)$, that $A$ is a ``corner'' does not preclude it from being CP (which might seem strange if one thinks about positive operators on the direct sum of two Hilbert spaces).  Indeed, it is easy to see that $1_N A 1_M = A$ exactly when there is a linear map $A' \colon M \to N$ with $A(x\oplus y) = 0 \oplus A'(x)$ for $x\in M, y\in N$, and then $A$ is CP if and only if $A'$ is.  We can define a Schur product on $\mc B(L^2(M), L^2(N))$ by $A'\star B' = m_N(A' \otimes B')m_M^*$, and then $A$ is Schur idempotent if and only if $A'$ is.  We hence have the following equivalent definition.

\begin{definition}
A \emph{Quantum adjacency operator} from $M$ to $N$ is a completely positive map $A \colon M \to N$ which is Schur idempotent.
\end{definition}

We shall work with this second definition, but the first, equivalent, definition allows us to quickly apply the existing theory, from the situation when $M=N$.  Rather than go through the (slightly tedious) derivation of taking direct sums, and then corners, we shall directly state the results in this more general situation, occasionally indicating how they follow from the direct-sum formulation.

At this point, we need to use the modular automorphism group of $\varphi_M$, for which we fix some further notation.  We want to choose a reference trace on $M$, for which the \emph{Markov Trace} $\Tr_M$ is most useful.  This is the unique trace on $M$ such that, if we form the GNS space $L^2(M)$ and GNS representation $\pi_M \colon M \to \mc B(L^2(M))$, we have that $\Tr \circ \pi_M = \Tr_M$, where $\Tr$ is the canonical trace on $\mc B(L^2(M))$.  For more, see for example \cite[Lemma~7.12]{daws_quantum_graphs}, and Remark~\ref{rem:delta_forms} below.  There is a positive invertible operator $Q_M\in M$ with $\varphi_M(x) = \Tr_M(Q_M x)$ for each $x\in M$.  Then the modular automorphism group $(\sigma^M_t) = (\sigma_t)$, modular conjugation $J=J_M$ and modular operator $\nabla = \nabla_M$ can be defined by
\[ \sigma_t(x) = Q_M^{it} x Q_M^{-it}, \quad J\Lambda(x) = \Lambda(\sigma_{-i/2}(x^*)), \quad \nabla\Lambda(x) = \Lambda(Q x Q^{-1}) \qquad (x\in M). \]
Further remarks can be found in \cite[Section~5]{daws_quantum_graphs} for example.

Recall the discussion of the opposite algebra from Section~\ref{sec:notation}.
Following \cite[Section~5.4]{daws_quantum_graphs} we define a linear bijection
\[ \Psi' \colon \mc B(L^2(M), L^2(N)) \to N \otimes M^\op; \quad \rankone{\Lambda(x)}{\Lambda(y)} \mapsto x \otimes \sigma_{i/2}(y)^{*\op}. \]
This map is some sort of generalisation of the ``Choi Matrix'' construction, \cite[Theorem~2]{Choi_CPMapsMatrices}.

\begin{theorem}\label{thm:Psi'_CP_idem_to_proj}
The map $\Psi'$ gives a bijection between maps $A\colon M\to N$ and $e\in N\otimes M^\op$ for which the following are equivalent:
\begin{enumerate}
    \item $A$ is completely positive and Schur-idempotent;
    \item $A$ is \emph{real} (meaning $A(x^*)=A(x)^*$ for $x\in M$) and Schur-idempotent; 
    \item $e$ is a (self-adjoint) projection.
\end{enumerate}
\end{theorem}
\begin{proof}
For completeness, we provide some details.  $\Psi'$ is a bijection.  For $a\in M$ let $m_M^*\Lambda(a) = \sum_i \Lambda(a_i) \otimes \Lambda(b_i)$.
Then
\begin{align*}
(\rankone{\Lambda(x_1)}{\Lambda(y_1)})  &   \star (\rankone{\Lambda(x_2)}{\Lambda(y_2)}) \Lambda(a) \\
&= m_N\big( \sum_i \Lambda(x_1) (\Lambda(y_1)|\Lambda(a_i)) \otimes \Lambda(x_2) (\Lambda(y_2)|\Lambda(b_i) \big) \\
&= \Lambda(x_1x_2) (m_N(\Lambda(y_1) \otimes \Lambda(y_2)|\Lambda(a))
= \Lambda(x_1x_2) (\Lambda(y_1y_2)|\Lambda(a)),
\end{align*}
so $(\rankone{\Lambda(x_1)}{\Lambda(y_1)}) \star (\rankone{\Lambda(x_2)}{\Lambda(y_2)}) = \rankone{ \Lambda(x_1x_2) }{ \Lambda(y_1y_2) }$.  As $y \mapsto \sigma_{i/2}(y)^{*\op}$ is a homomorphism, it follows that $\Psi'$ is a homomorphism, for the Schur product.

Our map $\Psi'$ is exactly the restriction of the map $\Psi' \colon \mc B(L^2(M)\oplus L^2(N)) \to (M\oplus N)\otimes (M\oplus N)^\op$ from \cite[Section~5.4]{daws_quantum_graphs}.  This latter map is a bijection between CP maps and positive elements.  Thus, $A$ being CP and Schur-idempotent is equivalent to $e$ being idempotent and positive, that is, a projection.  The equivalence with $A$ being real is \cite[Proposition~2.23]{matsuda_class_m2}, see also \cite[Theorem~5.37]{daws_quantum_graphs}.
\end{proof}

\begin{remark}
Let us say a few more words about the ``real'' condition.  If we consider the larger class of \emph{completely bounded} maps, then there is a natural period 2 anti-linear map $T\mapsto T^\dagger$ defined by $T^\dagger(a) = T(a^*)^*$.  For example, we implicitly see this occurring in the ``Paulsen $2\times 2$ matrix trick'', \cite[Theorem~8.3]{PaulsenBook}.  This map is not an involution for the composition product: indeed, $(T\circ S)^\dagger = T^\dagger \circ S^\dagger$.  However, we see that $A$ is real exactly when $A^\dagger = A$.

We claim that $\dagger$ is an involution for the Schur product, and that $\Psi'$ is a $*$-homomorphism.  It suffices to show that $\Psi'(A^\dagger) = \Psi'(A)^*$ when $A=\rankone{x}{y}$.  We have that
\begin{align*}
A^\dagger(a) = A(a^*)^* = \overline{ (\Lambda(y)|\Lambda(a^*)) } x^* 
= \varphi(ay) x^* = \varphi(\sigma_i(y)a) x^* = \rankone{x^*}{\sigma_i(y)^*} (a),
\end{align*}
where we again identify maps $L^2(M)\to L^2(N)$ with maps $M\to N$.  It follows that $\Psi'(A^\dagger) = x^* \otimes \sigma_{i/2}(\sigma_i(y)^*)^{*\op}
= x^* \otimes \sigma_{i/2}(y)^{\op} = \Psi'(A)^*$ as claimed.

One could also show this directly by first observing that $m^\dagger = m \circ \tau$ where $\tau$ is the tensor swap map, and that $(m^*)^\dagger = \tau m^*$.  It follows that $(A_1\star A_2)^\dagger = m^\dagger(A_1^\dagger \otimes A_2^\dagger)(m^*)^\dagger = m(A_2^\dagger \otimes A_1^\dagger)m^* = A_2^\dagger \star A_1^\dagger$.

We also note that $1^\dagger=1$ and so the ``reflexive'' axiom for a Quantum Graph (see \cite[Defintion~2.4]{daws_quantum_graphs} for example), that $A \star 1 = 1$, is equivalent to $1\star A^\dagger=1$ which is of course $1\star A=1$, if $A$ is real.
\end{remark}

We quickly explore the motivating example when both algebras are commutative.

\begin{example}\label{eg:adj_in_comm_case}
Let $X,Y$ be finite sets, and let $M = \ell^\infty(X), N=\ell^\infty(Y)$, acting on $\ell^2(X), \ell^2(Y)$ respectively.  A quantum relation $V$ from $M$ to $N$ is hence a $\ell^\infty(Y)$-$\ell^\infty(X)$-bimodule $V\subseteq\mc B(\ell^2(X), \ell^2(Y))$.  By considering the action of the minimal idempotents $(e_x)_{x\in X}$ and $(e_y)_{y\in Y}$, we see that $V$ is the linear span of the matrix elements $e_{y,x}$ which it contains.  That is, setting $R = \{ (y,x) : e_{y,x}\in V\} \subseteq Y\times X$, we have that $R$ is a relation, and $V = \lin \{ e_{y,x} : (y,x)\in R \}$.  We saw such ideas in Example~\ref{eg:classical_channel_relation}.

The projection $e\in N\otimes M^\op \cong \ell^\infty(Y\times X)$ is simply the indicator function of $R$.  Our choice of the Markov trace is the natural one on $\ell^\infty(X)$, that induced by the counting measure on $X$.  Then $\Lambda \colon \ell^\infty(X) \to \ell^2(X)$ is the formal identity on functions, and the same for $N$, and so the map $\Psi'$ is
\[ \Psi' \colon \mc B(\ell^2(X), \ell^2(Y)) \to \ell^\infty(Y\times X); \quad e_{y,x} = |\Lambda(e_y)\rangle \langle\Lambda(e_x)| \mapsto e_y \otimes e_x = e_{(y,x)}, \]
where $e_{(y,x)} \in \ell^\infty(Y\times X)$ is the minimal projection onto the singleton $(y,x)$.  So $A$, thought of as a matrix, has a $1$ in the $(y,x)$ position when $(y,x)\in R$, and $0$ otherwise.  Of course, $A$ is then (classically) Schur-idempotent and completely positive.
\end{example}

Given $\varphi_M$, the natural faithful positive funtional on $M^\op$ is given by $\varphi_{M^\op}(x^\op) = \varphi(x)$, for which the GNS space $L^2(\varphi_{M^\op})$ is unitarily equivalent to the conjugate space $\overline{ L^2(\varphi_M) }$ for $\Lambda^\op(x^\op) = \overline{\Lambda(x^*)}$, and where $x^\op$ acts as $x^\top \in \mc B(\overline{ L^2(\varphi_M) })$, compare \cite[Lemma~5.32]{daws_quantum_graphs}.

The algebra $N\otimes M^\op$ acts naturally on $L^2(N) \otimes L^2(M^\op) = L^2(N) \otimes \overline{L^2(M)}$. Then projections $e\in N\otimes M^\op$ biject with subspaces $V_0 \subseteq L^2(N) \otimes \overline{L^2(M)}$ which are $(N\otimes M^\op)' = N' \otimes M'^\op$ invariant.  Indeed, $V_0$ is just the image of $e$.  We regard $L^2(N) \otimes \overline{L^2(M)}$ as $HS(L^2(M), L^2(N))$ the Hilbert--Schmidt operators, which here equals $\mc B(L^2(M), L^2(N))$.  Under this identification, $V_0$ becomes a $N'$-$M'$-bimodule, that is, a quantum relation.

However, as suggested by Wasilewski in \cite{Wasilewski_Quantum_Cayley}, we typically do not associate $e$ directly with $V_0$: there are various motivations for this, \cite{daws2025quantumgraphsinfinitedimensionshilbertschmidts, Wasilewski_Quantum_Cayley} also Remark~\ref{rem:can_we_remove_mod} below.  Instead, associate an adjacency operator $A$ with projection $e$, and quantum relation $V$, as follows.  From \cite[Section~2.4]{daws2025quantumgraphsinfinitedimensionshilbertschmidts} we have that when $A = \sum_j \rankone{\Lambda(x_j)}{\Lambda(y_j)}$ then
\begin{equation}
T\in V \quad\iff\quad 
\sum_j \sigma_{-i/4}(x_j) T \sigma_{-i/4}(y_j^*) = T.
\label{eq:what_in_V_from_A}
\end{equation}
Then, by \cite[Proposition~2.9]{daws2025quantumgraphsinfinitedimensionshilbertschmidts}, see also \cite[Theorem~A]{Wasilewski_Quantum_Cayley}, $V^*$ corresponds to $\tau(e)^\op$, the tensor swap of $e$, and to the KMS adjoint of $A$, which by \cite[Lemma~6.21]{daws2025quantumgraphsinfinitedimensionshilbertschmidts} is $J_M A^* J_N$, where $A^*$ is the Hilbert space adjoint of $A$.  (The tensor swap map is $N\otimes M^\op \to M^\op \otimes N$ but we wish to land in $M\otimes N^\op$, which is why we write $\tau(e)^\op$ here.)
Notice that $V^*$ corresponding to $\tau(e)^\op$ seems very natural, and wouldn't be true (in general) if we used $V_0$ in place of $V$.

We wish to consider operators which are not necessarily idempotent, so we now (in a small way) generalise some results from \cite[Section~2]{daws2025quantumgraphsinfinitedimensionshilbertschmidts}.  Given any $A\in \mc B(L^2(M), L^2(N))$ we can still consider $x = \Psi'(A) \in N \otimes M^\op$ and hence form the subspace $\im(x) \subseteq HS(L^2(M), L^2(N))$.  The following defines a $V$ linked to $A$ and $x$, and checks that when $x$ is a projection, we recover the relationship \eqref{eq:what_in_V_from_A}.

\begin{proposition}\label{prop:image_links}
Let $A\in \mc B(L^2(M), L^2(N))$ be $A = \sum_j \rankone{\Lambda(x_j)}{\Lambda(y_j)}$ and set $x = \Psi'(A)$.  Define
\[ V = \{ \nabla_N^{1/4} T \nabla_M^{1/4} : T \in \im(x) \} \subseteq \mc B(L^2(M), L^2(N)). \]
Then:
\begin{enumerate}[(1)]
    \item\label{prop:image_links:1}
    $\im(x) = \big\{ \sum_j x_j T \sigma_{i/2}(y_j)^* : T \in HS(L^2(M), L^2(N)) \big\}$;
    \item\label{prop:image_links:3}
    $V = N' \nabla^{1/4} A \nabla^{-1/4} M' = \lin\{ a' \nabla^{1/4} A \nabla^{-1/4} b' : a'\in N', b'\in M' \}$ where here we consider $A \in \mc B(L^2(M), L^2(N))$.  In particular, $V$ is a quantum relation from $M$ to $N$.
    \item\label{prop:image_links:4}
    we have that $\im(\tau(x^*)^{\op}) = \{ T^* : T\in\im(x) \}$, and that the quantum relation and adjacency operator associated with $\tau(x^*)^\op$ are $V^*$ and $\nabla^{-1/2} A^* \nabla^{1/2}$ respectively. 
    \item\label{prop:image_links:2}
    when $x$ is idempotent, \eqref{eq:what_in_V_from_A} holds.
\end{enumerate}
\end{proposition}
\begin{proof}
We have that $x = \Psi'(A) = \sum_j x_j \otimes \sigma_{i/2}(y_j)^{*\op}$.  Thus a typical member of $\im(x)$ is
\[ \sum_j x_j(\xi) \otimes \overline{ \sigma_{i/2}(y_j)\eta }
\cong \sum_j \rankone{ x_j(\xi) }{ \sigma_{i/2}(y_j)\eta }
= \sum_j x_j \rankone{ \xi }{ \eta } \sigma_{i/2}(y_j)^*, \]
for $\xi\in L^2(N), \eta\in L^2(M)$.  Taking the linear span shows \ref{prop:image_links:1}.

Now set $\xi = \Lambda(a), \eta = \Lambda(b)$ for some $a\in N, b\in M$.  Then a typical member of $V$ is
\begin{align*}
\sum_j  &  \nabla^{1/4} x_j \rankone{ \Lambda(a) }{ \Lambda(b) } \sigma_{i/2}(y_j)^* \nabla^{1/4}
= \sum_j \sigma_{-i/4}(x_j) \nabla^{1/4} \rankone{ \Lambda(a) }{ \Lambda(b) } \nabla^{1/4} \sigma_{-i/4}(y_j^*)  \\
&= \sum_j \rankone{ \Lambda(\sigma_{-i/4}(x_j) \sigma_{-i/4}(a) }{ \Lambda( \sigma_{i/4}(y_j) \sigma_{-i/4}(b) ) } \\
&= \sum_j J \sigma_{-i/4}(a^*) J \rankone{ \Lambda(\sigma_{-i/4}(x_j) }{ J \sigma_{-i/4}(b^*) J \Lambda( \sigma_{i/4}(y_j) ) }  \\
&= J \sigma_{-i/4}(a^*) J \nabla^{1/4} \sum_j \rankone{ \Lambda(x_j) }{ \Lambda(y_j) } \nabla^{-1/4} J \sigma_{i/4}(b) J
= a' \nabla^{1/4} A \nabla^{-1/4} b',
\end{align*}
say, where $a' = J \sigma_{-i/4}(a^*) J\in M'$ and $b' = J \sigma_{i/4}(b) J \in N'$.  Here we used that $\sigma_{i/2}(y_j)^* = \nabla^{1/2} y_j^* \nabla^{-1/2}$ and so forth.
As any member of $M'$ (respectively, $N'$) can arise, taking linear spans shows \ref{prop:image_links:3}.

For \ref{prop:image_links:4}, we note that $\tau(x)^{*\op} = \sum_j \sigma_{i/2}(y_j) \otimes x_j^{*\op}$ and so $\im(\tau(x)^{*\op}) = \{ \sum_j \sigma_{i/2}(y_j) T x_j^{*} : T\in HS(L^2(N), L^2(M)) \} = \{ T^* : T \in\im(x) \}$ by comparison with \ref{prop:image_links:1}.  The associated $V$ is hence $\{ \nabla^{1/4} T \nabla^{1/4} : T^*\in\im(x)\} = V^*$.  The associated quantum adjacency operator is $\sum_j \rankone{ \Lambda(\sigma_{i/2}(y_j) }{ \Lambda(\sigma_{-i/2}(x_j)) } = \nabla^{-1/2} A^* \nabla^{1/2}$.

We can re-write the right-hand-side of \eqref{eq:what_in_V_from_A} as
\begin{align*}
& \sum_j x_j \nabla^{-1/4} T \nabla^{1/4} y_j^* = \nabla^{-1/4} T \nabla^{1/4} \\
\iff\quad & \sum_j x_j \nabla^{-1/4} T \nabla^{-1/4} \sigma_{-i/2}(y_j^*) = \nabla^{-1/4} T \nabla^{-1/4}
\end{align*}
because $\sigma_{-i/4}(x_j) = \nabla^{1/4} x_j \nabla^{-1/4}$, and so forth.  We now recognise the action of $x$, so this condition is equivalent to $\nabla^{-1/4} T \nabla^{-1/4}$ begin a fixed point of $x$.  When $x$ is idempotent, this is equivalent to $\nabla^{-1/4} T \nabla^{-1/4} \in \im(x)$, by definition, that $T\in V$, hence verifying \ref{prop:image_links:2}.
\end{proof}

\begin{remark}\label{rem:nabla_or_Q}
As in \cite[Remark~2.7]{daws2025quantumgraphsinfinitedimensionshilbertschmidts}, when defining $V$, and in point \ref{prop:image_links:3}, we can replace usage of $\nabla_N$ by $Q_N$, and $\nabla_M$ by $Q_M$.
\end{remark}

\begin{remark}\label{rem:delta_forms}
So far we have proceeded quite abstractly, but as $M$ is finite-dimensional, of course $M$ is just a direct-sum of matrix algebras, say $M = \bigoplus_{\alpha\in I} \mathbb M_{n(\alpha)}$ for some finite set $I$.  Let $1_\alpha$ be the unit of the factor $\mathbb M_{n(\alpha)}$, so $\{ 1_\alpha \}$ is the collection of minimal central idempotents of $M$.  Let $(e^\alpha_{i,j})_{i,j=1}^{n(\alpha)}$ be the matrix units of $\mathbb M_{n(\alpha)}$.  By \cite[Lemma~7.12]{daws_quantum_graphs}, the Markov trace is defined by $\Tr_M(x) = \sum_\alpha n(\alpha) \Tr(1_\alpha x)$ where $\Tr$ is the non-normalised trace on $\mathbb M_{n(\alpha)}$.  Adapting the calculation from before \cite[Definition~2.9]{daws_quantum_graphs}, for example, one can check that $\Tr_M$ gives that $mm^*=1$.

Some further remarks can be found in \cite[Proposition~2.1]{Banica_Sym_Gen_Coaction}, but be aware that in this reference Banica assumes that all traces are normalised, which is not our convention.
\end{remark}

We now calculate explicitly a link between the action of $A$ and $\Psi'(A)$; this should be compared to \cite[Section~6]{daws2025quantumgraphsinfinitedimensionshilbertschmidts}, but we give here a direct calculation.

\begin{lemma}\label{lem:action_A_to_Psi'A}
Let $A \colon M \to N$, and let $\Psi'(A) \in N\otimes M^\op$ act on $L^2(N) \otimes \overline{L^2(M)}$.  Then for $b_0,b_1\in M$,
\[ \big( \xi_1 \big| A(b_1^*b_0) \xi_0 \big)
= \big( \xi_1\otimes \overline{J\Lambda(b_1)} \big| \Psi'(A)(\xi_0 \otimes\overline{J\Lambda(b_0)}) \big)
\qquad (\xi_0,\xi_1 \in L^2(N)). \]
\end{lemma}
\begin{proof}
Let $x = \Psi'(A) =  x_j \otimes \sigma_{i/2}(y_j)^{*\op}$, say, so $A = \sum_j \rankone{\Lambda(x_j}{\Lambda(y_j)}$, meaning that $A\colon M \to N$ is $b \mapsto A(b) = \sum_j \varphi(y_j^*b) x_j$.  For $b_0, b_1\in M$ and $\xi_0,\xi_1\in L^2(N)$, we have
\begin{align*}
& \big( \xi_1\otimes \overline{J\Lambda(b_1)} \big| \Psi'(A)(\xi_0 \otimes\overline{J\Lambda(b_0)}) \big)
= \sum_j (\xi_1 | x_j \xi_0 ) ( \overline{J\Lambda(b_1)} | \overline{\sigma_{i/2}(y_j)J\Lambda(b_0)} ) \\
&= \sum_j (\xi_1 | x_j \xi_0 ) (\Lambda(b_1) | J\sigma_{i/2}(y_j)J\Lambda(b_0) ) 
= \sum_j (\xi_1 | x_j \xi_0 ) (\Lambda(b_1) | \Lambda(b_0 \sigma_{-i}(y_j^*)) )   \\
&= \sum_j (\xi_1 | x_j \xi_0 ) \varphi(b_1^* b_0 \sigma_{-i}(y_j^*))
= \sum_j (\xi_1 | x_j \xi_0 ) \varphi(y_j^* b_1^* b_0)
= (\xi_1 | A(b_1^*b_0) \xi_0),
\end{align*}
as claimed.
\end{proof}

\subsection{CP maps to quantum relations revisited}\label{sec:CP_to_QR_again}

Now is an appropriate point to discuss (some of) Verdon's work from \cite{Verdon_CovQuantumCombs}.\footnote{The reader is warned that some diagrams in \cite{Verdon_CovQuantumCombs} are corrupted, for example (14), and so the arXiv version is more readable, arXiv:2302.07776 [math.OA] see \url{https://arxiv.org/pdf/2302.07776}.}
The setting in that paper is an abstract $C^*$-$2$-Category: the motivation is that one can take this category to be $\rep(G)$ for some compact (quantum) group $G$, and then constructions are automatically covariant with respect to the group action.  The algebra objects here are ``separable standard Frobenius algebras'', \cite[Appendix~A]{Verdon_CovQuantumCombs}: if we take our category to be just $2\Hilb$, then we recover finite-dimensional $C^*$-algebras $M$, and the ``separable'' condition corresponds to a choice of trace with $m^*m=1$, that is, what we call the Markov Trace $\Tr_M$.  In this abstract setting, it is not clear what ``Hilbert space'' an algebra acts on, so Verdon's definition of a quantum relation, \cite[Definition~3.2]{Verdon_CovQuantumCombs}, is the analogue of the projection $e = \Psi'(A)\in N\otimes M^\op$ considered above, as $e$ is defined just using the algebra structure.  Given a CP map $A \colon M \to N$ (compare \cite[Definition~2.6]{Verdon_CovQuantumCombs}) the ``underlying quantum relation'', \cite[Definition~3.2]{Verdon_CovQuantumCombs} is, in our language, the support projection of $x=\Psi'(A)$.  As $A$ CP is equivalent to $x$ being positive, the support projection is just the projection onto the image of $x$.  As effectively we work here only in the tracial situation, we see that we obtain exactly the same relation between $A$ and $V = \im(x)$ as in Proposition~\ref{prop:image_links}.

It is shown in \cite[Proposition~3.5]{Verdon_CovQuantumCombs} that the map $A \to V$ is a full, dagger-preserving functor.  Let us give an analogue of this in our setting, where $\varphi_M,\varphi_N$ can be arbitrary.  We give the category of CP maps the dagger structure of the KMS adjoint, $A \mapsto A^*_{KMS} = \nabla^{-1/2} A^* \nabla^{1/2} = J A^* J$, as discussed before.  Proposition~\ref{prop:image_links}\ref{prop:image_links:4} shows that $A^*_{KMS}$ is CP, because $\Psi'(A^*_{KMS}) = \tau(\Psi'(A))^{\op *}$ is positive.

\begin{proposition}\label{prop:functor_A_to_V}
The map $A \to V$ given by Proposition~\ref{prop:image_links} is a functor which preserves the adjoint, and which is full (surjective on hom sets).
\end{proposition}
\begin{proof}
Proposition~\ref{prop:image_links}\ref{prop:image_links:4} shows that this map is $*$-preserving, and Proposition~\ref{prop:image_links}\ref{prop:image_links:3} shows that when $A=\id_M$ we have that $V=M'$, these being the identity morphisms of the relevant categories.  Proposition~\ref{prop:composition_of_adj_ops} below shows that given $A_1\colon M_1\to M_2$ and $A_2\colon M_2\to M_3$ CP maps, with associated quantum relations $V_1, V_2$, we have that $A_2\circ A_1$ gives the quantum relation $V_2\circ V_1$.  We hence have a functor.

Given any quantum relation $V$, let $e$ be the projection onto $\nabla_N^{-1/4} V \nabla_M^{-1/4}$, and let $A = \Psi'^{-1}(e)$, so by Theorem~\ref{thm:Psi'_CP_idem_to_proj}, $A$ is in particular CP, and by Proposition~\ref{prop:image_links}, $V$ is associated to $A$.  So every $V$ arises from some $A$, showing fullness.
\end{proof}

We now have two ways to associate a CP map with a quantum relation: either consider $\theta\colon N\to M$ and form $V^\theta$ as in Section~\ref{sec:channels_to_relations}, or consider $A\colon M \to N$, thought of as a member of $\mc B(L^2(M), L^2(N))$, form $\Psi'(A)$, and form $V$.  These are obviously different, but it seems possible that if $\theta = A^*_{KMS}$ (so at least (co)domains are correct) then we might get that $V^\theta$ is associated to $\Psi'(A)$.  We shall see that this is true in the tracial situation, but the general case is more complicated.

\begin{proposition}\label{prop:A_vs_theta}
Let $A, V$ be linked as in Proposition~\ref{prop:image_links}, with $A$ completely positive, and set $\theta = A^*_{KMS} \colon N \to M$, which is also CP.  Then $V^\theta = \nabla_N^{1/4} V \nabla_M^{-1/4}$.  In particular, when $\varphi_M, \varphi_N$ are traces, $V^\theta = V$.
\end{proposition}
\begin{proof}
Use Proposition~\ref{prop:image_links} to associate $A$ with the positive operator $\Psi'(A)$ and quantum relation $V = \nabla^{1/4} \im\Psi'(A) \nabla^{1/4}$.  By Proposition~\ref{prop:image_links}\ref{prop:image_links:4}, $A^*_{KMS}$ is associated with the positive operator $\tau(\Psi'(A))^{\op *}$.  Then Lemma~\ref{lem:action_A_to_Psi'A} applied to $A^*_{KMS}$ shows that for $a_0,a_1\in N, b_0,b_1\in M$,
\begin{equation}\label{eq:KMS_59}
\big( \Lambda(b_1) \big| A^*_{KMS}(a_1^*a_0) \Lambda(b_0) \big)
= \big( \Lambda(b_1) \otimes \overline{J\Lambda(a_1)} \big| \tau(\Psi'(A))^{\op *} (\Lambda(b_0) \otimes \overline{J\Lambda(a_0)}) \big).
\end{equation}
We aim to find a Stinespring dilation of $\theta$; again this is similar to \cite[Section~6]{daws2025quantumgraphsinfinitedimensionshilbertschmidts}, though we'll use \eqref{eq:KMS_59} directly.  Let $j \colon \overline{L^2(N)} \to L^2(N); \overline\xi \mapsto J\xi$, a unitary, as $J$ is an anti-linear involution.  Then for $a_0,a_1,y_0,y_1\in N$ we find that
\begin{align}
& ( a_1 j y_1^\op \overline{\Lambda(1)} | a_0 j y_0^\op \overline{\Lambda(1)} )
= ( a_1 J \Lambda(y_1^*) | a_0 J \Lambda(y_0^*) )
= ( \Lambda( a_1 \sigma_{-i/2}(y_1) ) | \Lambda( a_0 \sigma_{-i/2}(y_0) ) )  \notag\\
&= \varphi_N( \sigma_{i/2}(y_1^*) a_1^* a_0 \sigma_{-i/2}(y_0) )
= \varphi_N( a_1^* a_0 \sigma_{-i/2}(y_0)\sigma_{-i/2}(y_1^*) )
= ( \Lambda(a_1) | Jy_1y_0^*J \Lambda(a_0) )  \notag\\
&= ( y_0^* J \Lambda(a_0) | y_1^* J \Lambda(a_1) )
= ( y_1^\op \overline{J\Lambda(a_1)} | y_0^\op \overline{J\Lambda(a_0)} ).  \label{eq:j_conj}
\end{align}
Choose some $y\in N\otimes M^\op$ with $\Psi'(A) = y^*y$, so $\tau(\Psi'(A))^{\op*} = \tau(y)^\op \tau(y)^{\op*}$ because $x\mapsto \tau(x)^{\op *}$ is a homomorphism, and define
\[ u \colon L^2(M) \to L^2(M) \otimes L^2(N); \quad \xi \mapsto (1\otimes j)\tau(y)^{\op*}(\xi \otimes \overline{\Lambda_N(1)}). \]
This gives a dilation of $\theta$, because
\begin{align*}
\big( \Lambda(b_1) \big| &   u^*( 1 \otimes a_1^*a_0 )u \Lambda(b_0) \big)  \\
&= \big( (1\otimes a_1j) \tau(y)^{\op*}(\Lambda(b_1)\otimes\overline{\Lambda(1)}) \big| (1\otimes a_0j)\tau(y)^{\op*}(\Lambda(b_1)\otimes\overline{\Lambda(1)}) \big)  \\
&= \big( \tau(y)^{\op*} (\Lambda(b_1)\otimes\overline{J\Lambda(a_1)}) \big| \tau(y)^{\op*} (\Lambda(b_0)\otimes\overline{J\Lambda(a_0}) \big)  \\
&=\big( \Lambda(b_1) \big| A^*_{KMS}(a_1^*a_0) \Lambda(b_0) \big),
\end{align*}
from \eqref{eq:j_conj} and \eqref{eq:KMS_59}.

As we now have a dilation, Theorem~\ref{thm:UCP_to_QR} shows that
\[ V^\theta = \lin\{ y' (\langle \xi| \otimes 1) u : y'\in N', \xi\in L^2(M) \}. \]
Let $y = \sum_i y_i \otimes x_i^\op \in N\otimes M^\op$, choose $b\in M, a\in N$ and set $y' = JaJ \in N'$.  Then for $\xi\in L^2(M)$,
\begin{align*}
y'(\langle \Lambda(b)|\otimes 1)u(\xi)
&= \sum_i y' j (\langle \Lambda(b)|\otimes 1) (x_i^*\xi \otimes y_i^{*\op}\overline{\Lambda(1)})
= \sum_i y' j \overline{\Lambda(y_i)} (\Lambda(b)|x_i^*\xi)  \\
&= \sum_i JaJJ\Lambda(y_i) (\Lambda(x_ib)|\xi)
= \sum_i \rankone{ J\Lambda(ay_i) }{ \Lambda(x_ib) } \xi.
\end{align*}
As $J\Lambda(ay_i) = \sigma_{-i/2}(y_i^*) J\Lambda(a)$, we see that
\[ y'(\langle \Lambda(b)|\otimes 1)u = \sum_i \sigma_{-i/2}(y_i^*) \rankone{J\Lambda(a)}{\Lambda(b)} x_i^*,  \]
and so taking the linear span over $a,b$ shows that $V^\theta$ is the image of $\sum_i \sigma_{i/2}(y_i)^* \otimes x_i^{\op*} = (\sigma_{i/2}\otimes\id)(y)^*$ acting on $L^2(N) \otimes \overline{L^2(M)}$.  As $(\sigma_{i/2}\otimes\id)(y) = (\nabla^{-1/2}\otimes 1) y (\nabla^{1/2}\otimes 1)$, we see that
\begin{align*}
\im (\sigma_{i/2}\otimes\id)(y)^*
&= \ker (\sigma_{i/2}\otimes\id)(y) = \ker y (\nabla^{1/2}\otimes 1) = \ker (\nabla^{1/2}\otimes 1) y^*y (\nabla^{1/2}\otimes 1) \\
&= \im (\nabla^{1/2}\otimes 1) \Psi'(A) (\nabla^{1/2}\otimes 1)
= (\nabla^{1/2}\otimes 1) \im \Psi'(A),
\end{align*}
where here we consider $\im \Psi'(A) \subseteq L^2(N) \otimes \overline{L^2(M)}$.  If instead we regard this space as being $\mc B(L^2(M), L^2(N))$, we obtain $V^\theta = \nabla^{1/2} \im \Psi'(A) = \nabla^{1/2} \nabla^{-1/4} V \nabla^{-1/4} = \nabla^{1/4} V \nabla^{-1/4}$, as claimed.
\end{proof}

\begin{remark}\label{rem:can_we_remove_mod}
It would of course have been nice if $V^\theta = V$ instead of $\nabla_N^{1/4} V \nabla_M^{-1/4}$; we wonder if some of our ``choices'' has lead to this appearance of the modular operators.  In \cite[Section~5]{daws_quantum_graphs} we considered may ways to get a bijection between $\mc B(L^2(M))$ and $M\otimes M^\op$, compatible with the Schur product.  Only $\Psi' = \Psi'_{0,1/2}$ (or its symmetric counterpart $\Psi = \Psi_{1/2,0}$) give a bijection between completely positive $A$ and positive $x\in M\otimes M^\op$.  Once we are using $\Psi'$, considering $A^*_{KMS}$ is natural, because it preserves being CP, and comes from the natural ``involution'' $x\mapsto \tau(x)^{*\op}$.  The proof of Proposition~\ref{prop:A_vs_theta} shows that then $V^\theta = \nabla^{1/2} \im\Psi'(A)$.  Of course, we could \emph{define} the quantum relation associated to $A$ and/or $x=\Psi'(A)$ to be $\nabla^{1/2} \im(x)$, but then $A^*_{KMS}$ and/or $\tau(x)^{*\op}$ would not be associated to $V^*$: only the choice made in Proposition~\ref{prop:image_links} gives this.
\end{remark}

We can now quickly show ``functoriality'', see Proposition~\ref{prop:functor_A_to_V}.

\begin{proposition}\label{prop:composition_of_adj_ops}
Let $M_1, M_2, M_3$ be finite-dimensional von Neumann algebras, and let $V_1 \subseteq \mc B(L^2(M_1), L^2(M_2))$ and $V_2 \subseteq \mc B(L^2(M_2), L^2(M_3))$ be quantum relations, with associated adjacency operators $A_1, A_2$.  Set $V = V_2 \circ V_1 \subseteq \mc B(L^2(M_1), L^2(M_3))$, and let $x = \Psi'(A_2\circ A_1)$ a positive operator in $M_3 \otimes M_1^\op$.  Then $V$ is the quantum relation associated to $x$ by Proposition~\ref{prop:image_links}.
\end{proposition}
\begin{proof}
For $i=1,2$ set $\theta_i = (A_i)^*_{KMS}$ so Proposition~\ref{prop:A_vs_theta} shows that $V^{\theta_i} = \nabla^{1/4} V_i \nabla^{-1/4}$.  Let $A = A_2 \circ A_1$ so $A^*_{KMS} = \theta_1 \circ \theta_2$ and let $U$ be the quantum relation associated to $A$, so $V^{\theta_1\circ\theta_2} = V^{A^*_{KMS}} = \nabla^{1/4} U \nabla^{-1/4}$.  By Proposition~\ref{prop:cp_to_qr_composition} we have that $V^{\theta_1\circ\theta_2} = V^{\theta_2} \circ V^{\theta_1} = \nabla^{1/4} V_2 \nabla^{-1/4} \nabla^{1/4} V_1 \nabla^{-1/4} = \nabla^{1/4} V \nabla^{-1/4}$ and so we conclude that $U=V$ as claimed.
\end{proof}

\begin{remark}
The contention of Section~\ref{sec:quantum_functions}, in particular Proposition~\ref{prop:star_is_inverse}, is that $V\mapsto V^*$ is a sort of (partial) inverse operation; but here it seems related to the (KMS) adjoint on CP maps.  However, we need to be a little careful.  The functor from Proposition~\ref{prop:functor_A_to_V} is of course not an equivalence, and so just because two CP maps give the same relation, the maps might be different.  Compare Section~\ref{sec:adj_ops_from_hms} below, where we compute the Schur-idempotent $A$ associated to $V^\theta$, for a $*$-homomorphism $\theta$.
\end{remark}

Next we show that all quantum relations arise from a CP map; the following proof is very much in the vein of Verdon's approach.

\begin{corollary}\label{corr:all_fd_qrs_from_CP}
Let $M,N$ be finite-dimensional, and let $V$ be a quantum relation from $M$ to $N$.  There is a CP map $\theta \colon N \to M$ with $V = V^\theta$.
\end{corollary}
\begin{proof}
Let $\varphi_M = \Tr_M$, and let $M$ act on $L^2(M)$, with the same choices for $N$.  We realise $V$ as a subspace of $\mc B(L^2(M), L^2(N))$, and so by Proposition~\ref{prop:image_links} applied to, for example, the projection onto $V$, we obtain an CP $A\colon M \to N$.  Set $\theta = A^*_{KMS} = A^*$, and apply the preceding result.
\end{proof}

As we saw in Remark~\ref{rem:not_all_cosurj}, not all cosurjective $V$ arise from UCP $\theta$.  Related is the characterisation given by \cite[Proposition~3.6]{Verdon_CovQuantumCombs}: here ``channels'' are trace-preserving (TP) CP maps, not UCP maps, but see the following lemma.
Unfortunately, we don't follow the argument given by \cite[Proposition~3.6]{Verdon_CovQuantumCombs}; we make some comments.

\begin{lemma}\label{lem:tcpc_ucp}
Let $A \colon M \to N$ be CP.  Then $A$ is satisfies $\varphi_N\circ A = \varphi_M$ if and only if $A^*$ is UCP, if and only if $(\varphi_N\otimes\id)\Psi'(A) = 1^\op$.
\end{lemma}
\begin{proof}
The relation between $A$ and $A^*$ is that $\varphi_N(a^*A(b)) = \varphi_M(A^*(a)^*b)$ for each $a\in N, b\in M$.  As $\varphi_M$ is faithful, it follows that $A^*(1)=1$ if and only if $\varphi_M(b) = \varphi_N(A(b))$ for each $b$, that is $\varphi_M = \varphi_N\circ A$, i.e.\@ $A$ preserves the functionals.  As in Proposition~\ref{prop:image_links}, let $A = \sum_j \rankone{\Lambda(x_j)}{\Lambda(y_j)}$ so $\Psi'(A) = \sum_j x_j \otimes \sigma_{i/2}(y_j)^{*\op}$ and hence $(\varphi_N\otimes\id)\Psi'(A) = \sum_j \varphi_N(x_j) \sigma_{-i/2}(y_j^*)^{\op}$, while $\varphi_N(A(b)) = \sum_j \varphi_N(x_j) \varphi_M(y_j^*b)$ for $b\in M$.  So $\varphi_N\circ A = \varphi_M$ if and only if $\sum_j \varphi_N(x_j) y_j^* = 1$ if and only if $(\varphi_N\otimes\id)\Psi'(A) = 1^\op$ as $\sigma_{i/2}$ is a unital homomorphism.
\end{proof}

\begin{remark}
In our language, \cite[Proposition~3.6]{Verdon_CovQuantumCombs} claims that, with $e \in N \otimes M^\op$, we can choose an adjacency operator $A$ with $x=\Psi'(A)$ having support projection $e$ and $(\varphi_N\otimes\id)\Psi'(A) = 1^\op$, if and only if $(\varphi_N\otimes\id)(e)$ is invertible.  We show ``only if'': let $u = (\varphi_N\otimes\id)(e)$, so by finite-dimensionality, it suffices to show that $u$ is injective.  If $u\xi=0$ then $0 = (\xi|u\xi) = (\Lambda_N(1)\otimes\xi|e(\Lambda_N(1)\otimes\xi))$ so $\Lambda_N(1)\otimes\xi \in \ker(e) = \im(e)^\perp = \im(x)^\perp = \ker(x)$, so $0 = (\Lambda_N(1)\otimes\xi|x(\Lambda_N(1)\otimes\xi)) = (\xi|(\varphi_N\otimes\id)(x)\xi) = \|\xi\|^2$.

However, the ``if'' case does not hold.  Let $M=N=\mathbb M_2$, and let $N\otimes M^\op \cong \mathbb M_4$ act on $\mathbb C^4$.  Suppose $e = \rankone{\xi}{\xi}$ for some unit vector $\xi=(\xi_i)\in\mathbb C^4$, so $x$, being positive with the same image as $e$, must be a positive multiple of $e$.  Set $\xi_i = 1/2$ for $i\leq 3$, and $\xi_4 = i/2$.  Then the matrix entries are $e_{i,j} = \xi_i \overline{\xi_j} = 1/4$ excepting $e_{4,j} = i/4, e_{i,4} = -i/4$ for $i,j\leq 3$.  So
\[ u = (\Tr\otimes\id)(e) = \begin{pmatrix}
    1/2 & (1-i)/4 \\ (1+i)/4 & 1/2 \end{pmatrix}. \]
This has non-zero determinant, so is invertible, but as $x$ is a positive multiple of $e$, it is not possible that $(\Tr\otimes\id)(x)=1$.

It remains an interesting question if one can characterise when we can choose $x$ with $(\varphi_N\otimes\id)(x)=1^\op$ in terms of $e$.  Furthermore, the philosophy of this paper is to look at the quantum relation $V$, which here would correspond to giving a condition which used the image of $e$, as a subspace of $\mc B(L^2(M), L^2(N))$, directly.
\end{remark}

A quantum graph in the sense of Weaver \cite{Weaver_QuantumGraphs} is a quantum relation $V$ from $M$ to $M$ which is \emph{symmetric}, $V^*=V$, and \emph{reflexive}, $1\in V$ (equivalently, $M'\subseteq V$).  Verdon, \cite[Definition~3.9]{Verdon_CovQuantumCombs}, considers just the symmetric condition, while if we also have $1\in V$ then $V$ is said to be a ``quantum confusability graph''.  The graph defined by a CP map is $V^{\theta *} \circ V^\theta$, \cite[Proposition~3.11]{Verdon_CovQuantumCombs}, compared Remark~\ref{rem:graph_from_qr}.  The following is adapted from our understanding of how Verdon's proof of \cite[Proposition~3.12]{Verdon_CovQuantumCombs} works.  As ever, we consider UCP maps, not ``channels'' (TPCP maps).  We split the argument into two parts.

\begin{lemma}\label{lem:A_pos_can_square-root}
Let $M$ be a finite-dimensional von Neumann algebra, set $\varphi_M = \Tr_M$ the Markov Trace, and let $A\in\mc B(L^2(M))$ be a positive operator, such that the induced map $M\to M$ is CP.  Letting $S$ be the quantum relation associated to $A$, there is a Hilbert space $L$ and a CP map $\theta \colon \mc B(L) \to M$ with $V^{\theta *} \circ V^\theta = S$.
\end{lemma}
\begin{proof}
Let $x = \Psi'(A) \in M \vnten M^\op$ a positive element.  As $A$ is positive, there is a finite sequence $(a_i)_{i=1}^n$ in $M$ with $A = \sum_i \rankone{\Lambda(a_i)}{\Lambda(a_i)}$, so that $x = \sum_i a_i \otimes a_i^{*\op}$.  Define
\[ u\colon L^2(M) \to L^2(M) \otimes \mathbb C^n; \quad \xi \mapsto \sum_i a_i^*(\xi) \otimes \delta_i. \]
We ``rotate'' $u$ to find $v\colon \overline{\mathbb C^n} \to \overline{L^2(M)} \otimes L^2(M)$.  This means that if $u = \sum_j \rankone{\xi_j\otimes \alpha_j}{\eta_j}$ then $v = \sum_j \rankone{\overline{\eta_j}\otimes\xi_j}{\overline{\alpha_j}}$.  Choose $(e_k)$ in $M$ such that $(\Lambda(e_k))$ is an orthonormal basis of $L^2(M)$.  Then
\[ u(\xi) = \sum_{l,k} (\Lambda(e_k)|a_l^*(\xi)) \Lambda(e_k) \otimes \delta_l
\quad\implies\quad
u = \sum_{l,k} \rankone{\Lambda(e_k)\otimes\delta_l}{\Lambda(a_l e_k)}, \]
and so
\[ v(\overline{\delta_j}) = \sum_{l,k} \overline{\Lambda(a_l e_k)} \otimes \Lambda(e_k) ( \overline{\delta_l} | \overline{\delta_j} ) 
=\sum_{k} \overline{\Lambda(a_j e_k)} \otimes \Lambda(e_k). \]
The map $T\colon M \to \mc B(\overline{\mathbb C^n}); x \mapsto v^*(1\otimes x)v$ is CP, and satisfies
\begin{align*}
\big( \overline{\delta_i} \big| T(x) \overline{\delta_j} \big)
&= \sum_{k,l} \big( \overline{\Lambda(a_ie_k)} \otimes \Lambda(e_k) \big| \overline{\Lambda(a_je_l)} \otimes x\Lambda(e_l) \big)   \\
&= \sum_{k,l} (\Lambda(a_i^*a_je_l)|\Lambda(e_k)) (\Lambda(e_k)|x\Lambda(e_l))
= \sum_{l} ( \Lambda(a_i^*a_je_l) | x\Lambda(e_l) ) \\
&= \Tr(a_j^*a_ix) = \varphi(a_j^*a_ix),
\end{align*}
using that $\varphi = \Tr_M$ the Markov Trace on $M$.  We compute the adjoint $T^*$:
\begin{align*}
\varphi(x^* T^*(\rankone{\overline{\delta_i}}{\overline{\delta_j}}))
&= (\Lambda(x)|T^*(\rankone{\overline{\delta_i}}{\overline{\delta_j}}))_{L^2(M)}
= (T(x) | \rankone{\overline{\delta_i}}{\overline{\delta_j}} )_{HS}
= \Tr(T(x)^* \rankone{\overline{\delta_i}}{\overline{\delta_j}} ) \\
&= ( \overline{\delta_j} | T(x)^* \overline{\delta_i} )
= \overline{ \varphi(a_j^*a_ix) }
= \varphi(x^* a_i^* a_j).
\end{align*}
It follows that
\begin{align*}
(T^*\circ T)(x) = \sum_{i,j} \varphi(a_j^*a_ix) T^*( \rankone{\overline{\delta_i}}{\overline{\delta_j}} )
= \sum_{i,j} \varphi(a_j^*a_ix) a_i^* a_j
= \sum_{i,j} \rankone{ \Lambda(a_i^*a_j) }{ \Lambda(a_i^*a_j) } x.
\end{align*}
So $\Psi'(T^*\circ T) = \sum_{i,j} a_i^*a_j \otimes (a_i^*a_j)^{*\op} = x^*x$, and hence $\im \Psi'(T^*\circ T) = \im (x^*x) = \im(x)$ as $x$ is positive, and so $T^*\circ T$ gives the quantum relation $S$ we started with.  Set $\theta = T^*$ so Proposition~\ref{prop:A_vs_theta} shows that $V^\theta$ is the quantum relation induced by $T$, and by Proposition~\ref{prop:functor_A_to_V} we conclude that $T^*\circ T$ is associated to the quantum relation $V^{\theta *} \circ V^\theta$, that is, $S = V^{\theta *} \circ V^\theta$.
\end{proof}

\begin{proposition}[{\cite[Proposition~3.12]{Verdon_CovQuantumCombs}}]\label{prop:Verdon_All_QG_from_UCP}
Let $M$ be a finite-dimensional von Neumann algebra, and let $S$ be a quantum graph over $M$ (so $S^*=S$ and $1\in S$).  There is a UCP map $\theta\colon \mc B(L) \to M$, for some finite-dimensional Hilbert space $L$, with $V^{\theta *} \circ V^\theta = S$.  Hence every such $S$ is the quantum confusability graph of some UCP map to $M$.
\end{proposition}
\begin{proof}
Let $e \in M\otimes M^\op$ be the projection onto $S \subseteq \mc B(L^2(M)) \cong HS(L^2(M))$, and let $e_0$ be the projection onto $M' \subseteq \mc B(L^2(M))$.  As $M'\subseteq S$, we have that $e_0 \leq e$ and so $e_1 = e-e_0$ is a projection, orthogonal to $e_0$.  As $S^*=S$, we have that $\tau(e)^{*\op} = e$, and similarly $\tau(e_0)^{*\op} = e_0$, so also $e_1$ satisfies this.  Setting $A_1 = \Psi'^{-1}(e_1)$, we see that $A_1$ is CP, and that $A_1^* = A_1$.  A calculation shows that $\Psi'^{-1}(e_0) = 1_{L^2(M)}$.  As $A_1$ is self-adjoint, there is some $t>0$ with $A = 1 + tA_1$ a positive operator.  Choose $\theta$ using the lemma, so $V^{\theta *}\circ V^\theta$ is equal to the image of $\Psi'(A) = e_0 + te_1$.  As $e_0, e_1$ are orthogonal summing to $e$, we have that this image is $\im(e) = S$, as claimed.
\end{proof}

\begin{remark}
In the proof of Lemma~\ref{lem:A_pos_can_square-root}, we could have directly defined $v$ from the family $(a_i)$, but this would have further obscured the motivation from \cite{Verdon_CovQuantumCombs}.  The map $u$ dilates a CP map, say $\theta_A \colon M \to M; x \mapsto u^*(x\otimes 1)u = \sum_i a_i x a_i^*$, which is exactly the map $\theta_A$ from \cite[Proposition~2.4]{daws2025quantumgraphsinfinitedimensionshilbertschmidts} (this map having occurred in many places in the literature, going back to \cite{Weaver_QuantumRelations}).  In \cite[Section~6.5]{daws2025quantumgraphsinfinitedimensionshilbertschmidts} we were sceptical that $\theta_A$ being considered as a CB (or even CP, as here) map could be profitable: this now seems hasty!  However, we remark that in the application in Proposition~\ref{prop:Verdon_All_QG_from_UCP}, $A$ is \emph{not} the adjacency operator of the quantum graph $S$.
\end{remark}

We will now show an analogous characterisation of which quantum graphs (symmetric relations, not necessarily reflexive) arise from CP maps, not assumed unital.  This could be done by adapting the proof above, but this is complicated by the fact that for the analogy of $e_0$, it is hard to compute $\Psi'^{-1}(e_0)$.  We instead use some of the theory we have developed.

\begin{theorem}\label{thm:QG_from_CP}
Let $M$ be a finite-dimensional von Neumann algebra, and let $S$ be a quantum relation over $M$ with $S=S^*$.  The following are equivalent:
\begin{enumerate}[(1)]
    \item\label{thm:QG_from_CP:one}
    there is a von Neumann algebra $N$ and a CP map $\theta \colon N\to M$ with $S = V^{\theta *} \circ V^\theta$;
    \item\label{thm:QG_from_CP:two}
    there is $x_0\geq 0$ in $S \cap M$, such that for every $x=x^*\in S$ there is $t>0$ with $-tx_0 \leq x \leq tx_0$;
\end{enumerate}
In this case, we can take $x_0 = \theta(1)$, and we may take $N$ to be finite-dimensional.
\end{theorem}
\begin{proof}
We let $M$ act on $H$, so $S\subseteq\mc B(H)$.  If we have a CP map as in \ref{thm:QG_from_CP:one}, then let $u\colon H \to K \otimes L$ dilate $\theta$, where $N\subseteq \mc B(K)$, and set $x_0 = u^*u = \theta(1) \in M$.  For $x=x^*\in S$ there is $y\in N'\vnten\mc B(L)$ with $x = u^*yu$, so with $z = \frac12(y+y^*)$ we have that $z=z^*$ and $x = u^*zu$.  As $-\|z\| \leq z \leq \|z\|$ we find that $-\|z\| x_0 \leq x \leq \|z\| x_0$, hence showing \ref{thm:QG_from_CP:two}.

Let \ref{thm:QG_from_CP:two} hold and set $f_0$ to be the support projection of $x_0$.  Let $x=x^*\in S$ and by rescaling suppose that $-x_0 \leq x \leq x_0$.  There are $\epsilon_2 > 0$ with $x_0 \leq \epsilon_2 f_0$.  So $-\epsilon_2 f_0 \leq x \leq \epsilon_2 f_0$, so again by rescaling, we may suppose that $-f_0 \leq x \leq f_0$.  Then left and right multiplying by $1-f_0$ shows that $(1-f_0)x(1-f_0)=0$.

Suppose $y\in\mc B(H)$ is positive with $(1-f_0)y(1-f_0)=0$.  The $C^*$-condition shows that $y^{1/2}(1-f_0)=0$ so $y(1-f_0)=0$ so $y = yf_0$ and hence also $y=y^*=f_0y$.  Applying this to $y=x+f_0$ shows that $f_0y = yf_0 = y$ so $xf_0 = f_0x = x$.  Given any $x\in S$, forming real and imaginary parts gives that also $xf_0 = f_0x =x$.

We shall be careful with notation: regard $f_0$ as an operator $H\to H$, and let $p_0 \colon H \to H_0 = f_0(H)$ be the corestriction, so $p_0^* \colon H_0 \to H$ is the inclusion.  Hence $p_0 p_0^* = 1$ while $p_0^*p_0 = f_0$.  Let $M_0 = p_0 M p_0^* \subseteq \mc B(H_0)$ a von Neumann algebra.  On $H_0$ we have that $x_0$ is invertible (by finite-dimensionality), and so we can set $S_0 = \{ x_0^{-1/2} p_0 x p_0^* x_0^{-1/2} : s\in S \}$ to see that $S_0$ is a quantum graph on $H_0$, with $S = \{ x_0^{1/2} p_0^* x p_0 x_0^{1/2} : x \in S_0 \}$.  There is hence some $N$ and a UCP $\theta_0 \colon N \to M_0$ with $V^{\theta_0 *} \circ V^{\theta_0} = S_0$.  Define $\theta_1 \colon M_0 \to M_0$ by $\theta_1(x) = x_0^{1/2} x x_0^{1/2}$, a CP map, with $V^{\theta_1} = M_0' x_0^{1/2}$.  By Proposition~\ref{prop:cp_to_qr_composition}, $V^{\theta_1 \circ \theta_0} = V^{\theta_0} \circ V^{\theta_1} = V^{\theta_0} x_0^{1/2}$.

Let $\iota \colon M_0 \to M$ be the inclusion, namely $\iota(y) = p_0^* y p_0$, a normal $*$-homomorphism, and let $V^\iota$ be the quantum function, 
\begin{align*}
V^\iota &= \{ v \in \mc B(H, H_0) : yv = v\iota(y) = v p_0^* y p_0\ (y\in M_0) \}   \\
&= \{ v \in \mc B(H, H_0) : p_0 x p_0^* v = v f_0 x f_0 \ (x\in M) \}  \\
&= \{ v \in \mc B(H, H_0) : f_0 x f_0 p_0^* v = p_0^* v f_0 x f_0 \ (x\in M) \},
\end{align*}
as $p_0 f_0 = f_0^*$ and so forth.  So $v\in V^\iota$ if and only if $p_0^*v \in (f_0Mf_0)' = M_0' \oplus \mc B(H_0^\perp)$, acting on $H = H_0 \oplus H_0^\perp$.  Equivalently, $v = y' p_0$ for some $y'\in M_0'$.  By \cite[Proposition~II.3.10]{TakesakiI}, $M_0' = p_0 M' p_0^*$, and so $V^\iota= p_0 M' f_0$.

Set $\theta = \iota \circ \theta_1 \circ \theta_0 \colon N \to M$ so $V^\theta = V^{\theta_0} x_0^{1/2} \circ p_0 M' f_0$ by these calculations.  As $x_0\in M$ commutes with $f_0$, we see that $V^\theta = V^{\theta_0} p_0 x_0^{1/2}$ and so $V^{\theta*} \circ V^\theta = x_0^{1/2} p_0^* S_0 p_0 x_0^{1/2} = S$, as required to show \ref{thm:QG_from_CP:one}.
\end{proof}

All of these arguments are rather finite-dimensional in nature.  There is a short proof of Proposition~\ref{prop:Verdon_All_QG_from_UCP} given in  \cite[Proposition~6.7]{kornell2026quantumgraphshomomorphisms}, which again works with TPCP maps not UCP maps.  The idea of \cite[Proposition~6.7]{kornell2026quantumgraphshomomorphisms} is to use Duan's argument to find a TPCP map $\phi \colon \mc B(H) \to \mc B(K)$ which gives $S$, and then to simply restrict this to $M \subseteq\mc B(H)$ say giving $\hat\theta \colon M \to \mc B(K)$; a small check is then required to check that this doesn't change $S$ (using here that $S$ is an $M'$-bimodule).  As in Lemma~\ref{lem:tcpc_ucp}, the associated UCP map $\theta \colon \mc B(K) \to M$ satisfies the relation
\[ \Tr_M(\theta(a)b) = \Tr(a\hat\theta(b)) = \Tr(a\iota(\phi(b))) \qquad (a\in\mc B(K), b\in M), \]
where $\iota$ is the inclusion $M\to\mc B(H)$.  So $\theta = \hat\iota \circ \hat\phi$.  Indeed, with $\Tr_M$ the Markov trace, one readily checks that $\hat\iota \colon \mc B(H) \to M$ is a conditional expectation: this might be as expected, because it gives a way to change the codomain of a UCP map from $\mc B(H)$ to $M$.  It is possible to adapt this conditional expectation idea to work more generally, but unfortunately, having a \emph{normal} conditional expectation $\mc B(H) \to M$ is a rather strong condition, compare \cite[Exercise~IX.4.1]{TakesakiII}, forcing $M$ to be a direct (possibly infinite) sum of matrix algebras.  It would be very interesting to have tools to allow us to push these ideas over to more general von Neumann algebras.

\subsection{Ordering}

In earlier sections, the (inclusion) ordering on quantum relations was very important.  We now look at what this translates to for quantum adjacency operators.  Let $A,e,V$ be related.  As the bijection between $V$ and $\im(e)$ from Proposition~\ref{prop:image_links} is order preserving, we obtain the usual order on projections: $e_1 \leq e_2$ if and only if $e_2 e_1 = e_1$, equivalently, $e_1 e_2 = e_1$.  This is then seen to be equivalent to $A_2 \star A_1 = A_1 \iff A_1 \star A_2 = A_1$.

However, as composition of relations is given by composition of adjacency operators, Proposition~\ref{prop:composition_of_adj_ops}, which of course may not preserve Schur idempotency, giving a practical condition for e.g.\@ $A$ to be cosurjective seems hard, and we shall not pursue this here.  However, we do have the following.

\begin{proposition}\label{prop:when_A_coinj}
Let $A \colon M \to N$ be a quantum adjacency operator.  The following are equivalent:
\begin{enumerate}[(1)]
    \item\label{prop:when_A_coinj:1}
    $A$ is coinjective;
    \item\label{prop:when_A_coinj:2} 
    $A \circ (JA^*J)$, regarded as a map on $L^2(N)$, is a member of $N'$;
    \item\label{prop:when_A_coinj:3}
    there is a central positive element $x_0\in N$ with $(A \circ (JA^*J))(x) = xx_0 = x_0^{1/2} x x_0^{1/2}$ for each $x\in N$.
\end{enumerate}
\end{proposition}
\begin{proof}
Let $V$ be the quantum relation associated to $A$, so by definition, $A$ is coinjective when $V$ is, that is, when $V\circ V^* \subseteq N'$.  By Proposition~\ref{prop:composition_of_adj_ops}, $V\circ V^*$ is the image of $A \circ (JA^*J)$ which by Proposition~\ref{prop:image_links}\ref{prop:image_links:3} means $V\circ V^* = N' \nabla^{1/4} ( A \circ (JA^*J) ) \nabla^{-1/4} N'$.  As $\nabla$ implements a one-parameter group on $N'$, we see that $N' \supseteq V\circ V^*$ if and only if
\begin{align*}
& \quad    N' \supseteq \nabla^{1/4} N' ( A \circ (JA^*J) ) N' \nabla^{-1/4} \\
\iff & \quad  \nabla^{-1/4} N' \nabla^{1/4} = N' \supseteq N' ( A \circ (JA^*J) ) N',
\end{align*}
and this is equivalent to $A \circ (JA^*J) \in N'$ (as $1\in N'$).
As any member of $N'$ is of the form $\Lambda(x) \mapsto \Lambda(xx_0)$ for some $x_0\in N$, and if the map $x\mapsto xx_0$ is positive, then $x_0 = 1x_0$ is positive, and for each $x\geq 0$ we have that $xx_0 = (xx_0)^* = x_0 x$, so taking the linear span of such $x$, we conclude that $x_0$ is central.  Conversely, $x\mapsto x_0^{1/2} x x_0^{1/2}$ is completely positive for any positive $x_0$.
\end{proof}

As coinjective $V$ correspond to $*$-homomorphisms, the next section gives an explicit, if slightly complicated, expression for such $A$.

\subsection{Quantum functions}\label{sec:adj_ops_from_hms}

We consider what Theorem~\ref{thm:funcs_to_HMs} says about quantum adjacency operators.  Let $\theta\colon N \to M$ be a $*$-homomorphism, where again $M,N$ are finite-dimensional acting on $L^2(M)$ and $L^2(N)$, respectively.  Consider the quantum relation $V^\theta$ as in Definition~\ref{prop:defn_V_theta}.  Our ultimate aim is to find the quantum adjacency operator associated with $V^\theta$.

\begin{proposition}\label{prop:hm_to_adj}
Let $M, N$ be finite-dimensional, let $\theta\colon N \to M$ be a $*$-homomorphism, and form the quantum relation $V^\theta \subseteq \mc B(L^2(M), L^2(N))$.
\begin{enumerate}[(1)]
\item\label{prop:hm_to_adj:one}
There is a linear bijection between $V^{\theta*}$ and $\{ \xi\in L^2(M) : \theta(1)\xi=\xi \}$ where $v^*\in V^{\theta*}$ bijects with $\xi = v^*\Lambda(1)$, and $\xi$ determines $v^*$ as $v^*\Lambda(y) = \theta(y)\xi$ for $y\in N$.
\item\label{prop:hm_to_adj:two}
Let $\hat\theta \colon L^2(N) \to L^2(M)$ be the map induced by $\theta$, so $\hat\theta\Lambda(y) = \Lambda(\theta(y))$ for $y\in N$.  The $M'$-$N'$-bimodule generated by $\hat\theta$ is $V^{\theta*}$.
\end{enumerate}
\end{proposition}
\begin{proof}
We have that $v\in V^\theta$ when $v\colon L^2(M) \to L^2(N)$ and $v^* y = \theta(y) v^*$ for $y\in N$.  With $\xi=v^*\Lambda(1)$ we see that $v^*\Lambda(y) = v^*y\Lambda(1) = \theta(y)\xi$ for $y\in N$, and so $v^*$, hence $v$, is determined by $\xi$.  Further, $\theta(1)\xi = \theta(1)v^*\Lambda(1) = v^*\Lambda(1)=\xi$.  Given any $\xi$ with $\xi=\theta(1)\xi$ define $t \colon \Lambda(y) \mapsto \theta(y)\xi$ for $y\in N$, so $v=t^* \colon L^2(M) \to L^2(N)$ satisfies $v^* y \Lambda(z) = t\Lambda(yz) = \theta(yz) \xi = \theta(y) t\Lambda(z)$ for $y,z\in N$, and so $v^*y = \theta(y)v^*$ for $y\in N$ and hence $v\in V^\theta$, with $v^*\Lambda(1) =\theta(1)\xi=\xi$.  So \ref{prop:hm_to_adj:one} holds.

As $JMJ=M'$ and so forth, for $a'\in M', b'\in N'$ there are $x_0\in M, y_0\in N$ with $a'\Lambda_M(b) = \Lambda_M(bx_0)$ and $b'\Lambda_N(y) = \Lambda_N(yy_0)$ for $y\in N, b\in M$.  Thus
\[ a' \hat\theta b' \Lambda(y) = a' \hat\theta \Lambda(yy_0) = a'\Lambda(\theta(y)\theta(y_0))
= \theta(y) \Lambda(\theta(y_0)x_0), \]
and so $v^* = a' \hat\theta b' \Lambda(y)\in V^{\theta *}$ is associated to the vector $\xi = \Lambda(\theta(y_0)x_0)$, as in part \ref{prop:hm_to_adj:one}.  As any $\xi\in L^2(N)$ with $\theta(1)\xi=\xi$ arises in this way, we have shown \ref{prop:hm_to_adj:two}.
\end{proof}

When $\varphi_N=\Tr_N$ the Markov trace, we have that $m_N m_N^* = 1$, see Remark~\ref{rem:delta_forms}, and so $\theta$ being a homomorphism means that $m_M(\hat\theta\otimes\hat\theta) = \hat\theta m_N$ and so $\hat\theta \star \hat\theta = \hat\theta m_N m_N^* = \hat\theta$, so $\theta$ is Schur-idempotent, and of course $\theta$ is CP.  If also $\varphi_M$ is a trace, then point \ref{prop:hm_to_adj:two} and Proposition~\ref{prop:image_links}\ref{prop:image_links:3} shows that $\hat\theta$ must be the adjacency operator associated to $V^{\theta *}$.  Hence also $J\hat\theta^*J = \hat\theta^*$ (as $\varphi_M$ is a trace) is the adjacency operator of $V^\theta$.

In the general case, Proposition~\ref{prop:functor_A_to_V} doesn't help us to find the adjacency operator, because here we are seeking the \emph{Schur idempotent} CP map $A$ which bijects with $V$, not simply some CP map.  However, Proposition~\ref{prop:A_vs_theta} can provide some guidance in the non-tracial setting.  Let $\theta\colon N \to M$ be our $*$-homomorphism, so we seek $\phi\colon N \to M$ CP with $V^\phi = \nabla^{1/4} V^\theta \nabla^{-1/4}$, as then the $A$ with $A^*_{KMS}=\phi$ is associated to $V^\theta$, and we might hope to adjust $A$ to be Schur idempotent.  By Remark~\ref{rem:nabla_or_Q} we can instead look for $V^\phi = Q_N^{1/4} V^\theta Q_M^{-1/4}$, and this then allows us to work with different Hilbert spaces to the $L^2$ spaces.  In particular, by adjusting $K$ with $N\subseteq\mc B(K)$, we may suppose there is $r\colon H\to K$ with $r\theta(y) = yr$ for $y\in N$, so $\theta(y) = r^*yr$ is a dilation.  Hence $V^\theta = N'r$.  Set $u = Q_N^{1/4} r Q_M^{-1/4}$ and define $\phi(y) = u^*yu$ for $y\in N$, so $V^\phi = N' u = N'  Q_N^{1/4} r Q_M^{-1/4} =  Q_N^{1/4} N' r Q_M^{-1/4} =  Q_N^{1/4} V\theta Q_M^{-1/4}$ as desired.  Notice that then $\phi(y) = Q_M^{-1/4} \theta(Q_N^{1/4}yQ_N^{1/4}) Q_M^{-1/4}$.

Motivated by this, and working with examples, we try defining
\begin{align}\label{eq:defn_A_for_Vthetaadj}
A(x) = Q_M^{-1/4} u^{1/2} \theta(Q_N^{1/4} x Q_N^{1/4}) u^{1/2} Q_M^{-1/4}   \qquad (x\in N),
\end{align}
where here $u\in M \cap \theta(N)'$ is some positive operator which we shall define shortly.  Then $A$ maps into $M$ and is completely positive.

\begin{lemma}\label{lem:when_u_inv}
Let $u$ be invertible when restricted to $\im\theta(1)$.  Then $M' \nabla_M^{1/4} A \nabla_{N}^{-1/4} N' = V^{\theta *}$.
\end{lemma}
\begin{proof}
Let $a'\in M'$ be given by $a'\Lambda(a)=\Lambda(ax_0)$ for some $x_0\in M$, and similarly let $b'\in N'$ be associated to $y_0\in N$.  Then for $x\in N$,
\begin{align*}
a' \nabla_M^{1/4} A \nabla_N^{-1/4} b' \Lambda(x)
&= a' \nabla_M^{1/4} A \Lambda(Q_N^{-1/4} x y_0 Q_N^{1/4}) \\
&= a' \nabla_M^{1/4} \Lambda\big(   Q_M^{-1/4} u^{1/2} \theta(Q_N^{1/4} Q_N^{-1/4} x y_0 Q_N^{1/4} Q_N^{1/4}) u^{1/2} Q_M^{-1/4}  \big) \\
&= \Lambda\big(  \theta( x) \theta( y_0 Q_N^{1/2}) u Q_M^{-1/2} x_0  \big),
\end{align*}
where here we used that $u \in \theta(N)'$.  Set $y_0 = Q_N^{-1/2}$ so we obtain the term $\theta(1)$.  As $\theta(1) u = u\theta(1)$, it follows that $\im\theta(1)$ is an invariant subspace for $u$, and our hypothesis is that $u$ is invertible on this subspace, so there is $w \in M$ with $uw = wu = \theta(1)$.  For any $z_0\in M$ set $x_0 = Q_M^{1/2} w z_0$, to obtain the map $\Lambda(x) \mapsto \theta(x) \Lambda(z_0)$.  By Proposition~\ref{prop:hm_to_adj}, for each $v^* \in V^{\theta *}$ there is $\xi\in L^2(M)$ with $v^*\Lambda(x) = \theta(x)\xi$.  As $\xi = \Lambda(z_0)$ for some $z_0\in M$, this completes the proof.
\end{proof}

It now seems necessary to explicitly recognise $N$ as a direct sum of matrix algebras, say $N = \bigoplus_\alpha \mathbb M_{n(\alpha)}$.  For $x\in N$ let $x = (x_\alpha)$ in this direct sum.  Let $(1_\alpha)$ be the family of minimal central idempotents in $N$, so $1_\alpha$ is the unit of $\mathbb M_{n(\alpha)}$.
For a given $\alpha$, the restriction of $\theta$ to $\theta_\alpha \colon \mathbb M_{n(\alpha)} \to \theta(1_\alpha) M \theta(1_\alpha) \subseteq \mc B(\im \theta(1_\alpha))$ is a unital $*$-homomorphism, and so $\im \theta(1_\alpha) \cong \mathbb C^{n(\alpha)} \otimes K_\alpha$ for some auxiliary (finite-dimensional) Hilbert space $K_\alpha$, with $\theta_\alpha(x) = x\otimes 1$ under this isomorphism.  Our element $u \in \theta(N)' \cap M$ restricts to $\im\theta(1_\alpha)$ giving an element which commutes with $\mathbb M_{n(\alpha)} \otimes 1$, and so is of the form $1\otimes u_\alpha$.

\begin{lemma}\label{lem:cond_u_SI}
We have that $A$ is Schur-idempotent if and only if, for each $\alpha$, the positive $u_\alpha \in \mc B(K_\alpha)$ satisfies
\[ (\Tr\otimes\id)\big( (Q_{N,\alpha}^{-3/4} \otimes 1) (\theta(1_\alpha) Q_M^{-1/2} \theta(1_\alpha)) (1\otimes u_\alpha) (Q_{N,\alpha}^{1/4} \otimes 1) \big) = 1. \]
Here we treat $\theta(1_\alpha) Q_M^{-1/2} \theta(1_\alpha)$ as a member of $\mc B(\im \theta(1_\alpha)) \cong \mathbb M_{n(\alpha)} \otimes \mc B(K_\alpha)$.
\end{lemma}
\begin{proof}
Let $e_{i,j}$ denote the matrix units of a matrix algebra $\mathbb M_n$, and let $(e^\alpha_{i,j})$ denote the matrix units of $\mathbb M_{n(\alpha)} \subseteq N$.
A formula for $m^*$ is computed before \cite[Definition~2.9]{daws_quantum_graphs}, for example, and this gives
\[ m_N^*(e^\alpha_{i,j}Q_N^{-1}) = \sum_k e^\alpha_{i,k}Q_N^{-1} \otimes e^\alpha_{k,j}Q_N^{-1}. \]
Notice that the linear span of elements of the form $e^\alpha_{i,j}Q_N^{-1}$ gives all of $N$, and that $e^\alpha_{i,j}Q_N^{-1} = e^\alpha_{i,j}Q_{N_\alpha}^{-1}$.  For such elements, we have
\begin{align}
& m_M(A\otimes A)m_N^*(e^\alpha_{i,j} Q_N^{-1})
= \sum_k A(e^\alpha_{i,k}Q_{N,\alpha}^{-1}) A(e^\alpha_{k,j}Q_{N,\alpha}^{-1}) \notag \\
&= \sum_k Q_M^{-1/4} u \theta(Q_{N,\alpha}^{1/4} e^\alpha_{i,k} Q_{N,\alpha}^{-3/4}) Q_M^{-1/2} u \theta(Q_{N,\alpha}^{1/4} e^\alpha_{k,j} Q_{N,\alpha}^{-3/4}) Q_M^{-1/4},
\label{eq:AstarA}
\end{align}
here again using that $u\in\theta(N)'$.
The middle of this expression is
\[ \sum_k \theta(e^\alpha_{i,k}) \theta(Q_{N,\alpha}^{-3/4}) \theta(1_\alpha) Q_M^{-1/2} \theta(1_\alpha) u \theta(Q_{N,\alpha}^{1/4}) \theta(e^\alpha_{k,j}). \]
As this is an element of $\theta(1_\alpha) M \theta(1_\alpha)$, we use the isomorphism, and work in $\mathbb M_{n(\alpha)} \otimes \mc B(K_\alpha)$, and so obtain
\begin{align}
&\sum_k (e^\alpha_{i,k} Q_{N,\alpha}^{-3/4} \otimes 1) \theta(1_\alpha) Q_M^{-1/2} \theta(1_\alpha) (1\otimes u_\alpha) (Q_{N,\alpha}^{1/4} e^\alpha_{k,j} \otimes 1) \notag \\
&= e^\alpha_{i,j} \otimes (\Tr\otimes\id)\big( (Q_{N,\alpha}^{-3/4} \otimes 1) (\theta(1_\alpha) Q_M^{-1/2} \theta(1_\alpha)) (1\otimes u_\alpha) (Q_{N,\alpha}^{1/4} \otimes 1) \big).  \label{eq:should_collapse}
\end{align}

We compare \eqref{eq:AstarA} to
\[ A(e^\alpha_{i,j} Q_N^{-1}) =  Q_M^{-1/4} u \theta(Q_{N_\alpha}^{1/4} e^\alpha_{i,j} Q_{N,\alpha}^{-3/4}) Q_M^{-1/4}. \]
It follows that $A\star A = A$ if and only if \eqref{eq:should_collapse} is equal to simply $e^\alpha_{i,j} \otimes 1$, which is exactly the condition claimed in the lemma statement.
\end{proof}

We now proceed to show that a suitable $u$ always exists, and show how to construct it.  For a given $\alpha$, set
\[ t = \theta(1_\alpha) Q_M^{-1/2} \theta(1_\alpha) = \sum_{i,j} e^\alpha_{i,j} \otimes t_{i,j} \in \mathbb M_{n(\alpha)} \otimes \mc B(K_\alpha) = \mc B(\im\theta(1_\alpha)), \]
say.  The condition from Lemma~\ref{lem:cond_u_SI} becomes
\begin{align*}
1 &=
\sum_{i,j} (\Tr\otimes\id)\big( (Q_{N,\alpha}^{-3/4} \otimes 1) (e^\alpha_{i,j} \otimes t_{i,j} u_\alpha) (Q_{N,\alpha}^{1/4} \otimes 1) \big) \\
&= \sum_{i,j} \Tr(Q_{N,\alpha}^{-3/4} e^\alpha_{i,j} Q_{N,\alpha}^{1/4})  t_{i,j} u_\alpha
= \sum_{i,j} (Q_{N,\alpha}^{-1/2})_{j,i}  t_{i,j} u_\alpha.
\end{align*}
So define
\[ v = \sum_{i,j} (Q_{N,\alpha}^{-1/2})_{j,i} t_{i,j} \in \mc B(K_\alpha). \]

\begin{lemma}\label{lem:v_is_pos_inv}
The operator $v$ is positive and invertible, and $1\otimes v$ is the restriction of an element of $M$ to $\im\theta(1_\alpha)$.
\end{lemma}
\begin{proof}
Let $q = Q_{N,\alpha}^{-1/2}$ a positive matrix, so for $\xi\in K_\alpha$ we have
\[ (\xi|v\xi) = \sum_{i,j} q_{j,i} \big( \xi \big| t_{i,j} \xi \big)
= \sum_{i,j} q_{j,i} (\delta_i\otimes\xi|t(\delta_j\otimes\xi))
= \sum_{i,j} q_{j,i} t^\xi_{i,j}, \]
say, where $t^\xi$ is a positive matrix, because for $\eta\in\mathbb C^{n(\alpha)}$,
\[ (\eta | t^\xi(\eta))
=\sum_{i,j} \overline{\eta_i} (\delta_i\otimes\xi|t(\delta_j\otimes\xi)) \eta_j
= (\eta\otimes\xi|t(\eta\otimes\xi)) \geq 0, \]
as $t$ is positive.  So 
\[ (\xi|v\xi) = \Tr(qt^\xi) = \Tr(q^{1/2} t^\xi q^{1/2}) \geq 0, \]
and we conclude that $v$ is positive.

This calculation also shows that if $v\xi=0$ then $\Tr(q^{1/2}t^\xi q^{1/2})=0$ so $(t^\xi)^{1/2} q^{1/2}=0$, so $t^\xi=0$, as $q$ is invertible.  Then $(\eta\otimes\xi|t(\eta\otimes\xi))=0$ for each $\eta$, so $(\theta(1)(\eta\otimes\xi)|Q_M^{-1/2}\theta(1)(\eta\otimes\xi)) = 0$ and hence $\theta(1)(\eta\otimes\xi)=0$, for each $\eta$.  However, $\eta\otimes\xi \in \im\theta(1)$, and so we conclude that $\xi=0$.  So $v$ is injective, and hence invertible.

For any $i,j,k$, using the isomorphism $\im\theta(1_\alpha) \cong \mathbb M_{n(\alpha)} \otimes \mc B(K_\alpha)$, we have
\begin{align*}
&    
\theta(e^\alpha_{k,i}) Q_M^{-1/2} \theta(e^\alpha_{j,k})
= (e_{k,i}\otimes 1) t (e_{j,k}\otimes 1) = e_{k,k} \otimes t_{i,j}  \\
\implies &
1\otimes v = \sum_{i,j,k} e_{k,k} \otimes (Q_{N,\alpha}^{-1/2})_{j,i} t_{i,j}
= \sum_{i,j,k} (Q_{N,\alpha}^{-1/2})_{j,i} \theta(e^\alpha_{k,i}) Q_M^{-1/2} \theta(e^\alpha_{j,k}),
\end{align*}
which is a member of $M$ (or more accurately, the restriction of this element to $\im\theta(1_\alpha)$.)
\end{proof}

Given the lemma, we may hence define $u_\alpha$ to be the inverse of $v$, for each $\alpha$.
The projections $\theta(1_\alpha)$ are mutually orthogonal and sum to $\theta(1)$.  We can hence make sense of $u = \sum_\alpha \theta(1_\alpha) (1\otimes u_\alpha) \theta(1_\alpha)$ in the obvious way.  By the lemma, we see that $u\in M$, while by construction, $u$ commutes with $\theta(N)$.  As each $u_\alpha$ is positive, also $u$ is positive.  Our notation is consistent in that restricting $u$ back to $\im\theta(1_\alpha)$ does give the operator $1\otimes u_\alpha$, which is invertible.  Thus $u$ satisfies the conditions of Lemmas~\ref{lem:when_u_inv} and~\ref{lem:cond_u_SI}, and hence $A$ is the adjacency operator associated to $V^{\theta *}$.

It remains an interesting question if we can somehow ``compute'' the adjacency operator $A$ associated with a general $V^\theta$, where $\theta$ is only CP.  Thinking about Section~\ref{sec:CP_to_QR_again}, this is very related to computing the support projection of a positive operator.  Even in the classical situation, the task seems hard to do except by very explicit calculation.

\newcommand{\scr}{\mc}
\bibliographystyle{plain}
\bibliography{thebib.bib}

\end{document}